\documentclass[reqno,11pt]{amsart}
\usepackage{amsmath, latexsym, amsfonts, amssymb, amsthm, amscd}
\usepackage{multirow}

\usepackage{xcolor}

\usepackage[utf8]{inputenc} 

\usepackage{graphics,epsf,psfrag}
\numberwithin{equation}{section}

\newtheorem{theorem}{Theorem}[section]
\newtheorem{lemma}[theorem]{Lemma}
\newtheorem{proposition}[theorem]{Proposition}
\newtheorem{cor}[theorem]{Corollary}
\newtheorem{rem}[theorem]{Remark}

\newcommand{\R}{\mathbb{R}}

\newcommand{\N}{\mathbb{N}}

\newcommand{\cP}{{\ensuremath{\mathcal P}} }

\newcommand{\cC}{{\ensuremath{\mathcal C}} }

\newcommand{\cL}{{\ensuremath{\mathcal L}} }

\newcommand{\cW}{{\ensuremath{\mathcal W}} }

\newcommand{\bP}{{\ensuremath{\mathbf P}} }
\newcommand{\bE}{{\ensuremath{\mathbf E}} }

\DeclareMathSymbol{\leqslant}{\mathalpha}{AMSa}{"36} 
\DeclareMathSymbol{\geqslant}{\mathalpha}{AMSa}{"3E} 
\DeclareMathSymbol{\eset}{\mathalpha}{AMSb}{"3F}     
\renewcommand{\leq}{\;\leqslant\;}                   
\renewcommand{\geq}{\;\geqslant\;}                   
\newcommand{\dd}{\,\text{\rm d}}             

\def\restriction#1#2{\mathchoice
	{\setbox1\hbox{${\displaystyle #1}_{\scriptstyle #2}$}
		\restrictionaux{#1}{#2}}
	{\setbox1\hbox{${\textstyle #1}_{\scriptstyle #2}$}
		\restrictionaux{#1}{#2}}
	{\setbox1\hbox{${\scriptstyle #1}_{\scriptscriptstyle #2}$}
		\restrictionaux{#1}{#2}}
	{\setbox1\hbox{${\scriptscriptstyle #1}_{\scriptscriptstyle #2}$}
		\restrictionaux{#1}{#2}}}
\def\restrictionaux#1#2{{#1\,\smash{\vrule height 1\ht1 depth .85\dp1}}_{\,#2}}

\newcommand{\bbE}{{\ensuremath{\mathbb E}} }

\newcommand{\bbP}{{\ensuremath{\mathbb P}} }

\newcommand{\bbR}{{\ensuremath{\mathbb R}} }

\newcommand{\bbT}{{\ensuremath{\mathbb T}} }

\newcommand{\bbZ}{{\ensuremath{\mathbb Z}} }

\newcommand{\norm}[1]{\left\lVert#1\right\rVert}
\newcommand{\gnorm}[1]{\left\lVert#1\right\rVert_{\infty \to 1}}

\makeatletter
\def\captionfont@{\footnotesize}
\def\captionheadfont@{\scshape}

\long\def\@makecaption#1#2{%
  \vspace{2mm}
  \setbox\@tempboxa\vbox{\color@setgroup
    \advance\hsize-6pc\noindent
    \captionfont@\captionheadfont@#1\@xp\@ifnotempty\@xp
        {\@cdr#2\@nil}{.\captionfont@\upshape\enspace#2}%
    \unskip\kern-6pc\par
    \global\setbox\@ne\lastbox\color@endgroup}%
  \ifhbox\@ne 
    \setbox\@ne\hbox{\unhbox\@ne\unskip\unskip\unpenalty\unkern}%
  \fi
  \ifdim\wd\@tempboxa=\z@ 
    \setbox\@ne\hbox to\columnwidth{\hss\kern-6pc\box\@ne\hss}%
  \else 
    \setbox\@ne\vbox{\unvbox\@tempboxa\parskip\z@skip
        \noindent\unhbox\@ne\advance\hsize-6pc\par}%
\fi
  \ifnum\@tempcnta<64 
    \addvspace\abovecaptionskip
    \moveright 3pc\box\@ne
  \else 
    \moveright 3pc\box\@ne
    \nobreak
    \vskip\belowcaptionskip
  \fi
\relax
}
\makeatother
\def\writefig#1 #2 #3 {\rlap{\kern #1 truecm
\raise #2 truecm \hbox{#3}}}

\title[Long time dynamics for interacting  oscillators on graphs]{Long time dynamics for interacting  oscillators \\ on graphs}

\author{Fabio Coppini}
\address{
  Universit\'e Paris Diderot, Sorbonne Paris Cit\'e,   Laboratoire de Probabilit{\'e}s Statistique et Mod\'elisation, UMR 8001,
            F- 75205 Paris,France
}

\begin{document}
	\maketitle
	
\begin{abstract}
		The stochastic Kuramoto model defined on a sequence of graphs is analyzed: the emphasis is posed on the relationship between the mean field limit, the connectivity of the underlying graph and the long time behavior. We give an explicit deterministic condition on the sequence of graphs such that, for any finite time and any initial condition, even dependent on the network, the empirical measure of the system stays close to the solution of the McKean-Vlasov equation associated to the classical mean field limit. Under this condition, we study the long time behavior in the subcritical and in the supercritical regime: in both regimes, the empirical measure stays close to the (possibly degenerate) manifold of stable stationary solutions, up to times which can diverge as fast as the exponential of the size of the system, before Large Deviation phenomena take over. The condition on the sequence of graphs is derived by means of Grothendieck's Inequality and expressed through a concentration in $\ell_{\infty}\to \ell_1$ norm. It is shown to be satisfied by a large class of graphs, random and deterministic, provided that the average number of neighbors per site diverges, as the size of the system tends to infinity.
		 \\
		 \\
		 \textit{2010 Mathematics Subject Classification:} 60K35, 82C20, 82C31, 82C44.
		 \\
		 \\
		 \textit{Keywords and phrases:} Interacting oscillators, Long time dynamics, Kuramoto model, Random Graphs, Stochastic partial differential equations, Cut-norm, Grothendieck's Inequality, Self-normalized processes
\end{abstract}

\section{Introduction}
\subsection{Synchronization of mean field systems on graphs}
In recent years, synchronization of complex networks has become a very important topic for explaining real world phenomena. While in the physics literature the analysis has been pushed quite far and several extended reviews are available (e.g. \cite{cf:DB14,cf:RPJK16}), from a mathematical point of view these studies and the associated numerical simulations, can be regarded more as heuristic arguments than conclusive proofs.

The mathematical community has started working on particle systems on (random) graphs from the statistical mechanics point of view in the equilibrium regime and, with respect to the graph setting, assuming a locally tree-like structure (e.g \cite{cf:DM10}). Only in the last few years the attention has been focused on the dynamics of weakly interacting particles, tackling mean field systems on graphs, and their relationship with the corresponding thermodynamical limit (e.g. \cite{cf:BBW, cf:DGL}). These results, and the one presented here, are obtained for graphs in an intermediate regime between the sparse and the dense case, i.e. if $G_n$ has $n$ vertices and $np_n$ represents the average number of edges, then $1 \ll n p_n \leq n$. In the case of sparse graphs, i.e. $np_n = O(1)$, the limiting system seems to show a different phenomenology (\cite{cf:LRW,cf:ORS}).

Today, many results on the behavior of the empirical measure of such systems are available (\cite{cf:BBW,cf:CDG,cf:DGL,cf:L18,cf:RO}), but there is no agreement on the \emph{weakest} hypothesis the class of graphs should satisfy in order to obtain the classical mean field limit. It turns out that, depending on the setting one is considering, i.e. the normalization chosen in the interaction and/or the hypothesis on the initial data, different requirements on the graph may be asked.

To the author's knowledge, there exists no result on the longtime dynamics of a system defined on a sequence of graphs and the question whether the network is influencing the dynamics on long time scales, is still open and very much awaited with regards to applications.

\medskip

In this work, we attack these issues by considering a well known model of synchronization defined on a sequence of graphs: we consider the Kuramoto model (e.g. \cite{cf:AKur}) for which an extensive literature is available and many tools have now been developed (\cite{cf:BGP10,cf:BGP14,cf:GPP12}). For the sake of clarity, we study the model without the natural frequencies but our techniques apply as well in the quenched setting. We look for a result of mean field type with the minimal hypothesis on the initial conditions, i.e. the weak convergence of the empirical measure only, and by proposing a (deterministic) condition on the sequence of graphs which is shown to be satisfied by a large class of \emph{homogeneous} graphs, including Erd\H{o}s-Rényi random graphs with diverging average degree.

Finally, we show that the condition on the graph is not only sufficient for the system to converge to the mean field limit on bounded time intervals, but also that it is enough to study it on longer time scales. Namely, we push our analysis to the Large Deviation barrier of exponential time scales showing that, if the system synchronizes, then it keeps synchronized for long times.

\subsection{The model}
For each $n\in \N$, let $\xi^{(n)}$ be the adjacency matrix of a graph $(V^{(n)}, E^{(n)})$ with $n$ vertices:
\begin{equation}
	\label{d:graphSeq}
	V^{(n)}= \left\{1, \dots, n\right\}, \quad E^{(n)} = \left\{ (i,j) \in V^{(n)}\times V^{(n)} : \xi^{(n)}_{ij} \geq 1 \right\}.
\end{equation}
We consider both directed and undirected graphs as well as multigraphs so that $\xi^{(n)}_{ij}$ can take values in $\{1,\dots,n\}$ and not need to be equal to $\xi^{(n)}_{ji}$. We denote the corresponding (multi)graph by $\xi^{(n)}$ itself. Together with $\xi^{(n)}$, we consider a \emph{dilution} parameter $p_n \in (0,1]$ representing the average density of neighbors per site. The two quantities will be coupled so that it is useful to think of them as one single object, we refer to Subsection \ref{ss:graph} for the precise condition we required on it.

\medskip

Given  $\left( \xi^{(n)}, p_n \right)$, let $\{\theta^{i,n}_\cdot \}_{i=1,\dots,n}$ be the family of oscillators on $\bbT^n := \left(\bbR/2\pi \bbZ\right)^n$, which satisfy:
\begin{equation}
\label{eq:k}
\begin{cases}
	\dd{\theta}_{t}^{i,n} =\,  \frac{1}{n p_n} \sum_{j=1}^{n} \, \xi^{(n)}_{ij} \, J (\theta_{t}^{i,n} - \theta_{t}^{j,n}) \dd t + \dd B_{t}^{i}, \quad \text{for }t>0,\\
	\; \theta_0^{i,n} =\, \theta_0^i, \quad \text{ for }i \in  \{1,\dots,n\},
\end{cases}
\end{equation}
where $J(\cdot) = -K \sin (\cdot)$ with $K \geq 0$. Denote by $\bP$ the law induced by $\left\{B_\cdot^i\right\}_{i\in\N}$ which are independent and identically distributed (IID) Brownian motions on $\bbT$ and by $\left\{\theta_0^i \right\}_{i\in\N}$ the initial conditions. We consider both deterministic and random initial data and, whenever they are random, they have to be independent of the Brownian motions.

If $\{\xi^{(n)}_{ij}\}_{i,j}$ are symmetric, i.e. $\xi^{(n)}_{ij} = \xi^{(n)}_{ji}$ for $1 \leq i < j \leq n$, then the model is reversible (e.g. \cite{cf:BGP10}) with respect to the probability measure on $\bbT^n$ given by
\begin{equation}
	\pi^{(n)} (\dd \theta) = \frac 1{Z^{(n)}} \exp \left( - \frac K n \sum_{i,j=1}^n \xi^{(n)}_{ij} \cos(\theta^i - \theta^j) \right) \lambda_n (\dd \theta),
\end{equation}
where $Z^{(n)}$ is the normalizing constant and $\lambda_n$ the uniform probability measure on $\bbT^n$.

\medskip

The main quantity of interest in system \eqref{eq:k} is the empirical measure $\mu^n_t$ associated to $\{\theta^{i,n}_t\}_{i=1,\dots,n}$ and it is defined for all $t \geq 0$ by
\begin{equation}
\label{def:emp}
\mu^n_t := \frac 1n \sum_{j=1}^{n} \delta_{\theta^{j,n}_t} \in \cP(\bbT),
\end{equation}
the space of probability measure on the torus being denoted by $\cP(\bbT)$.

\medskip

\subsection{The reversible Kuramoto model and its mean field limit}
When $\xi^{(n)}_{ij} =1$ for $1\leq i,j \leq n$ and $p_n \equiv 1$ for all $n\in\N$, i.e. $\xi^{(n)}$ is the complete graph, system \eqref{eq:k} becomes:
\begin{equation}
	\label{eq:kc}
	\begin{cases}
		\dd{\bar{\theta}}_{t}^{i,n} =\,  (J *\bar{\mu}^n_t) (\bar{\theta}_{t}^{i,n}) \dd t + \dd B_{t}^{i}, \quad \text{for }t>0,\\
	\; \bar{\theta}_0^{i,n} =\, \theta_0^i, \quad \text{ for }i \in  \{1,\dots,n\},
	\end{cases}
\end{equation}
where $\bar{\mu}^n_t:=\frac 1n \sum_{j=1}^n \delta_{\bar{\theta}^{j,n}_n}$ is the associated empirical measure and $*$ stands for the convolution. We refer to \eqref{eq:kc} as the reversible Kuramoto model (e.g. \cite{cf:BGP10}).

\medskip

It is well known (e.g.  \cite[Proposition 3.1]{cf:BGP10}) that for all fixed time $T$, $\bar{\mu}^n_{t \in [0,T]}$ seen as a continuous function over $\cP(\bbT)$, weakly converges in $\cC^0([0,T], \cP(\bbT))$ to a deterministic limit $\mu_\cdot \in \cC^0([0,T], \cP(\bbT))$ that is solution to the following partial differential equation (PDE):
\begin{equation}
\label{eq:PDE}
\begin{cases}
\partial_t \mu_t (\theta) =  \tfrac 12 \partial^2_\theta \mu_t(\theta) - \partial_\theta [\mu_t(\theta) (J*\mu_t)(\theta)], \quad \text{for } \theta \in \bbT , \; 0< t \leq T,\\
\restriction{\mu_t}{t=0} = \mu_0,
\end{cases}
\end{equation}
provided that $\mu^n_0$ weakly converges to $\mu_0$ in $\cP(\bbT)$. If $\mu_0$ does not have a density, than \eqref{eq:PDE} has to be intended in the weak sense; however the regularity properties of the Laplacian operator make $\mu_t$ smooth for all $t>0$ (see again \cite[Proposition 3.1]{cf:BGP10}). Equation \eqref{eq:PDE} is often called McKean-Vlasov or Fokker-Planck equation and we refer to its solution $\mu_\cdot$ as to the mean field limit of the diffusions solving \eqref{eq:kc}.

\medskip

We recall here the most important results on \eqref{eq:PDE}, without giving any proof but referring to \cite{cf:GPP12} (and references therein) where a complete analysis of the global dynamics is presented.

As for the mean field limit of the classical Kuramoto model, \eqref{eq:PDE} is known to admit a phase transition depending on the coupling strength $K$: in the subcritical regime, for $0\leq K<K_c:=1$, the particles behave as they were independently distributed on the circle; in the supercritical regime, for $K>1$,  they tend to synchronize around the same phase. We do not consider the critical case $K=1$, since it does not add anything to the purpose of this work.

More precisely, in the subcritical regime there is a unique stationary solution which corresponds to the incoherent state $\frac 1{2\pi}$, the uniform measure on the torus  (see \cite[Proposition 4.1]{cf:GPP12}). It is globally attractive and the linear operator around it has negative spectrum bounded away from zero: we will make use of this property showing that the fluctuations given by the graph structure are controlled for all times, whereas the random fluctuations given by the Brownian motions are not and will make the system escape from $1/2\pi$ after some (very long) time, i.e. a Large Deviation phenomenon.

In the supercritical regime, when $K>1$, there is a manifold of stable stationary solutions corresponding to the synchronous states of the oscillators $\{\theta^{i,n}_\cdot\}_{i=1,\dots,n}$ (see \cite[Subsection 4.3]{cf:GPP12} and \cite{cf:BGP10}). Up to a rotation, all stable stationary solutions of \eqref{eq:PDE} are given by
\begin{equation}
\label{d:solM}
	q(\theta) = \frac{\exp\{2Kr \cos (\theta)\}}Z,
\end{equation}
where $Z$ is the normalizing constant and $r=r(K)$ is the unique solution in $(0,1)$ of a fixed point equation $r = \Psi(2Kr)$, see \cite{cf:BGP10} for a explicit formula of $\Psi$. The parameter $r$ is often referred to as the degree of synchronization of the system: $r$ close to 0 indicates that the particles are scattered around the circle, $r$ close to 1 that they are almost fully synchronized. We just recall that whenever $K<1$, the fixed point equation has a unique solution $r=0$, which in \eqref{d:solM} boils down to the uniform measure $1/2\pi$, and whenever $K>1$ the value $r=0$ is still a solution but the corresponding measures solving \eqref{eq:PDE} are unstable so that we will not consider them.

\medskip

Let $K>1$ and $0 < r < 1$. Observe that system \eqref{eq:kc} (and also \eqref{eq:k}) is invariant under rotations, this property is maintained in the limit \eqref{eq:PDE} and the manifold of stationary solutions $M$ can be described as
\begin{equation}
	M = \left\{ q_\psi \, : \, q_\psi (\cdot) = q (\cdot - \psi), \, \psi \in \bbT \right\}.
\end{equation}
It is possible to show that, unless one starts from the unstable manifold
\begin{equation}
	U=\left\{ \mu \in \cP (\bbT) \, : \, \int_{\bbT} \exp (i \theta) \mu (\dd \theta) = 0  \right\},
\end{equation}
the measure $\mu_t$ solution to \eqref{eq:PDE} converges to some $q_\psi \in M$ as $t$ tends to infinity, the phase $\psi\in\bbT$ depending only on $\mu_0$. Since each $q\in M$ is a stationary solution, the dynamics of $\mu_t$ is fully characterized for all times $t$.

\medskip

\subsection{The graph's perspective} \label{ss:graph}
The aim of this work is to investigate the weakest assumptions on the sequences $\xi = \left\{\xi^{(n)} \right\}_{n \in \N}$ and $\{p_n\}_{n\in \N}$, such that the long time behavior of \eqref{eq:k} is well understood: in other words, whenever system \eqref{eq:k} is comparable to \eqref{eq:kc} or to the mean field limit \eqref{eq:PDE}, under a proper scale between size of the system $n$ and some horizon time $T_n$.

The normalization sequence $p_n$ has to be chosen such that the interaction term in \eqref{eq:k} makes sense. At least, this requires the assumption that the quantity
\begin{equation}
	\label{h:gnorm}
	\frac{1}{np_n} \sum_{j=1}^n \xi^{(n)}_{ij}
\end{equation}
is of order one, for almost each vertex $i$ in the graph.

\begin{rem}
	Observe that whenever \eqref{h:gnorm} converges to zero or diverges, one should look for a different normalization in order to obtain a proper limit. A control on \eqref{h:gnorm} is thus required to exclude degenerate cases, yet it cannot be sufficient for our purpose: whenever one considers a graph composed of two (or more) highly connected components, the degree of each vertex can be correctly defined, but one cannot expect the convergence of the empirical measure since the behavior on each component may differ, depending on the initial conditions! We refer to  \cite[Remark 1.2]{cf:CDG} and \cite[Remark 1.4]{cf:DGL} for concrete examples and a precise analysis from this perspective, see also Remark \ref{r:conCom} in the next section.
\end{rem}

\medskip

For $n \in \N$, define the normalized adjacency matrix $P^{(n)} = \{P^{(n)}_{ij} \}_{i,j=1,\dots,n}$ by
\begin{equation}
	P^{(n)}_{ij} := \, \frac{\xi^{(n)}_{ij}}{p_n}, \quad \text{  for } i,j=1,\dots,n.
\end{equation} 
Recall that we do not assume any symmetry on $\xi^{(n)}$ and that it can also represent a multigraph. Define $\mathbf{1}^{(n)}$ as the adjacency matrix associated to the classical mean field model, i.e. $\mathbf{1}^{(n)}_{ij} = 1$ for $i,j=1,\dots,n$. One would like to compare $P^{(n)}$ to $\mathbf{1}^{(n)}$.

\medskip

It turns out that a sufficient condition for what we aim at, is given by a control on the difference between $P^{(n)}$ and $\mathbf{1}^{(n)}$ through the $\ell_\infty \to \ell_1$ norm. This norm is defined for a matrix $G = \{G_{ij}\}_{i,j=1,\dots,n}$ as
\begin{equation}
	\gnorm{G} := \sup_{\norm{s}_\infty \leq 1} \norm{G s}_1 = \sup_{s,t \in \{-1,1\}^n} G s t^\top = \sup_{s_i,t_j \in \{-1,1\}} \, \sum_{i,j=1}^n G_{ij} s_i t_j.
\end{equation}
It has received a lot of attention in the last years: it appears in many applications in computer science (e.g. \cite{cf:HSS}) and it has been shown to be very useful in graphs concentration (e.g. \cite{cf:GV, cf:LLV, cf:O09}). Part of this success is because of the equivalence to the cut-norm (e.g. \cite{cf:AN06}) and, as already remarked in \cite{cf:GV,cf:RO}, of Grothendieck's Inequality, which is recalled hereafter.

\medskip

\begin{theorem}[Grothendieck's inequality, {\cite[Theorem 2.4]{cf:gro}}]
	\label{thm:gro} 
	Let $\left\{ a_{ij} \right\}_{i,j=1,\dots,n}$ be a $n\times n$ real matrix such that for all $s_i,t_j \in \{-1,1\}$
	\begin{equation}
	\sum_{i,j=1}^n a_{ij} s_i t_j \leq 1.
	\end{equation}
	Then, there exists a constant $K_R >0$, such that for every Hilbert space $(H, \langle \cdot, \cdot \rangle_H )$ and for all $S_i$ and $T_j$ in the unit ball of $H$
	\begin{equation}
	\label{eq:gro}
	\sum_{i,j=1}^n a_{ij} \langle S_i, T_j \rangle_H \leq K_R.
	\end{equation}
\end{theorem}

It is indeed thanks to this inequality that $\ell_\infty \to \ell_1$ norm turns out to be the natural choice for our setting: an important part of the proof (Lemmas \ref{lem:g^n_t} and \ref{lem:g^n_tPi}) consists in showing that the fluctuations due to the graph structure can be described by expressions like \eqref{eq:gro}, and thus controlled by $\gnorm{\cdot}$. 

\medskip

From now on, the only condition we require on $\left(\xi^{(n)}, p_n \right)_{n \in \N}$ is to satisfy:
\begin{equation}
	\gnorm{P^{(n)} - \mathbf{1}^{(n)}} = o(n^2),
\end{equation}
or, in other words,
\begin{equation}
	\label{h:graph}
	\lim_{n\to\infty} \; \sup_{s_i,t_j \in \{-1,1\}} \, \frac 1 {n^2} \sum_{i,j=1}^n \left(\frac{\xi_{ij}^{(n)}}{p_n} - 1 \right)s_i t_j  = 0.
\end{equation}

\medskip

In Proposition \ref{lem:ber} it is shown that Erd\H{o}s-Rényi random graphs with parameter $p_n$ satisfy condition \eqref{h:graph} almost surely, provided that $np_n \uparrow \infty$. We also provide a class of deterministic graphs, Ramanujan graphs, that satisfies \eqref{h:graph} (see Proposition \ref{p:ram}) and give some link with the theory of graphons. 

Appendix A presents such results and includes remarks on the relationship between condition \eqref{h:graph}, the degree condition \eqref{h:gnorm} and the connectivity of $\{\xi^{(n)}\}_{n\in\N}$.

\subsection{Set-up and notations}
The closeness between $\mu^n_t$ and $\mu_t$ is studied through a norm which controls the bounded Lipschitz (or 1-Wasserstein) distance between probability measures, in an appropriate class of weighted Hilbert spaces $H_{-1,w}$. This class is defined as follows.

Denote by $\cC_0^1(\bbT)$ the space of $\cC^1$ functions on the torus with zero mean and consider
\begin{equation}
	\cL_0^2 =  \left\{ f \in \cL^2 (\bbT) : \int_{\bbT} f = 0 \right\},
\end{equation}
with canonical scalar product $(u,v) := \int_\bbT uv$, for $u,v \in \cL^2_0$. Let $w \in \cC^1(\bbT,(0,\infty))$ and $V$ be the closure of $\cC_0^1(\bbT)$ with respect to the norm $\norm{\varphi}_{H_{1, 1/w}} = \sqrt{\int_{\bbT} \tfrac{(\varphi')^2}{w}}$ for $\varphi \in \cC^1_0 (\bbT)$. It is easy to see that $V$ is continuously and densely injected in $\cL^2_0$ (thanks to the compactness of $\bbT$ and Poincaré inequality). Moreover, one can define an inner product on $V$ which makes it an Hilbert space $H_{1,1/w} := (V, \langle \cdot, \cdot \rangle_{H_{1, 1/w}} )$ where $  \langle \varphi , \psi \rangle_{H_{1, 1/w}} = \int_{\bbT}\frac{ \varphi' \psi'}{w}$ for all $\varphi, \psi \in \cC^1_0 (\bbT)$. The dual space of $H_{1,1/w}$ is denoted by $H_{-1,w}$. Observe that if $u,v \in \cL^2_0$ and $v \in H_{1,1/w}$, then $u \in H_{-1,w}$ and
\begin{equation}
	u(v):= \langle u, v \rangle_{-1,1} = (u,v),
\end{equation}
where $\langle \cdot, \cdot \rangle_{-1,1}$ denotes the action of $H_{-1,w}$ on $H_{1, 1/w}$, we omit the weight $w$.

\medskip

The action of a probability measure $\mu$ on a test function $h$ is denoted by $\langle \mu, h \rangle = \int h \dd \mu $: of course whenever $u$ and $v$ are regular enough, one has $u(v)=\langle u, v \rangle_{-1,1} = \langle u, v \rangle = (u,v)$, where we have abused of notation, denoting the density of a probability measure by the probability measure itself. 

Finally, observe that different weights $w$ give equivalent norms so that whenever the geometry of the space is not important, we consider the case $w\equiv1$ and simply note $\norm{\cdot}_{-1}$.  More information about the construction of $H_{-1, \omega}$ are given in Appendix B.

Hereafter we drop the dependency on $\bbT$, i.e. we write $\cC^1_0$ instead of $\cC^1_0 (\bbT)$ and so on for the other spaces and integrals.

\section{Main results}
We present the results in three consecutive subsections: we start by the finite time behavior, then pass to the supercritical regime and, finally, the subcritical case.

In all results, the convergence of empirical measures is stated in the norm $\norm{\cdot}_{-1}$. It is not difficult to see that the difference of two probability measures belongs to $H_{-1}$ and that the distance induced on $\cP(\bbT)$ controls the bounded Lipschitz distance (or, equivalently, the 1-Wasserstein distance). These details are covered in Appendix B.

\medskip

Recall that throughout the paper, we only require  $\left(\xi^{(n)}, p_n \right)_{n \in \N}$ to satisfy condition \eqref{h:graph} and $\mu_0 \in \cP(\bbT)$, no independence between $\mu^n_0$ and $\xi^{(n)}$ is demanded.

\bigskip

\subsection{The finite time behavior}
We give the result and then comment it.

\medskip

\begin{theorem}
	\label{thm:finTim}
	Let $K\geq0$. Suppose that for  all $\varepsilon_0 > 0$
	\begin{equation}
	\label{h:mu0}
	\lim_{n\to \infty} \bP \left( \norm{\mu_0^n - \mu_0}_{-1} \leq 	\varepsilon_0 \right) = 1.
	\end{equation}
	Then, for every fixed time $T> 0$ and for every $\varepsilon> 0$
	\begin{equation}
	\lim_{n\to \infty} \bP \left( \sup_{t\in[0, T]} \norm{\mu_t^n - \mu_t}_{-1} \leq \varepsilon \right) = 1.
	\end{equation}
\end{theorem}
The finite time behavior of weakly interacting particle systems on graphs is already known under suitable hypothesis on the initial conditions and on the graph sequence, we refer to Subsection \ref{ss:lit} for a comprehensive literature on the subject. We decide to present Theorem \ref{thm:finTim} because, contrary to all the previous results, it does not require any independence between initial conditions and (the realization of) the sequence of graphs. In particular, even if one accurately assigns the initial conditions for each vertex, the mixing properties of the graph will shuffle all the information and make the empirical measure converge, loosing any memory of the initial coupling. This property is crucial for studying the longtime behavior as pointed out in the next subsections.

Observe that Theorem \ref{thm:finTim} implies the existence of a unique giant component in $\{\xi^{(n)}\}_{n\in\N}$, as pointed out in the next remark.

\begin{rem}
	\label{r:conCom}
	The result is independent of $K$. First observe that this implies the uniqueness of a giant component: if there are two, then one can accurately prepare the initial conditions so to obtain different behaviors on the twos and loose the proximity to \eqref{eq:PDE}. Secondly, with the same argument one deduces that the size of the giant component is asymptotically $n$, i.e. all but $o(n)$ vertices are connected. Finally, the existence comes from the fact that the system cannot synchronize on components of size $o(n)$, no matter the value of $K$. Lemma \ref{lem:connect} shows that condition \eqref{h:gnorm} indeed implies the existence of a giant component of size asymptotically $n$.
\end{rem}

\medskip

\subsection{Long time behavior in the supercritical regime}
In the supercritical regime, we suppose to be already close to the manifold $M$ at time 0. However, since we do not assume any independence between graph and initial data, this hypothesis can be weakened by requiring the initial condition $\mu_0$ to be in the domain of attraction of $M$, i.e. $\mu_0 \in \cP(\bbT) \setminus U$, and using Theorem \ref{thm:finTim}. One can then start after some time $T$ with initial condition given now by $\mu^n_T$ (and dependent on the graph!): if $T$ is big enough, than $\mu^n_T$ will be close to $M$. Observe that the choice of $T$ depends only on how close to $M$ $\mu_t$ has to be, it thus depends only on $\mu_0$.

\medskip

Before stating the theorem, we define the distance of a probability measure from $M$. For $\mu \in \cP(\bbT)$, let
\begin{equation}
\label{d:dist}
	\textup{dist} (\mu,M) := \inf_{\nu \in M} \norm{\mu-\nu}_{-1}.
\end{equation}
We are ready for the main result of this section.
\medskip

\begin{theorem}
	\label{thm:sup} Let $K> 1$. Suppose there exists $\psi \in \bbT$ such that for every $\varepsilon_0 > 0$
	\begin{equation}
	\lim_{n\to \infty} \bP \left( \norm{\mu_0^n - q_\psi}_{-1} \leq 	\varepsilon_0 \right) = 1.
	\end{equation}
	Then, for every positive sequence $\{T_n\}_{n\in\N}$ such that $T_n =\exp(o(n))$, and for all $\varepsilon > 0$ small enough
	\begin{equation}
	\lim_{n\to \infty} \bP \left( \sup_{t\in[0, T_n]}  \textup{dist}(\mu^n_t,M) \leq \varepsilon \right) = 1.
	\end{equation}	
\end{theorem}

Theorem \ref{thm:sup} implies the proximity of the empirical measure to the manifold of solutions of the McKean-Vlasov equation \eqref{eq:PDE} for almost exponential times. On this time scale, Large Deviation phenomena take control of the finite system (e.g. \cite{cf:DG,cf:FW,cf:OV}) making it escape from the stationary solutions.

Observe that we do not prove the closeness to the mean field limit $\mu_\cdot$. Indeed, it is by now well known that, on longtime scales, the mean field limit is not a faithful description of the finite system of $n$ diffusions. In other words, the behavior of $\mu^n_{T_n}$ highly depends on the scale of time $T_n$ under consideration, whereas the dynamics of $\mu_t$ is deterministic and completely known for large $t$, i.e. it sticks to $q_\psi$.

In \cite{cf:BGP14}, a deep analysis of the longtime dynamics for the classical mean field system \eqref{eq:kc} is presented. Namely, it is shown that $\mu_t$ solution to the PDE \eqref{eq:PDE} is a reasonable approximation of $\bar{\mu}^n_t$ for times scales of order $o(\log n)$. On times proportional to $n$, the dynamics of the empirical measure can be coupled to a Brownian motion on $M$ with a non trivial diffusion coefficient that can be explicitly computed (see \cite[Theorem 1.1]{cf:BGP14}). Whereas the PDE prescribes the system to stay synchronized on a fixed phase, the noise induced by the Brownian motions makes this phase oscillate and it turns out that the oscillations become significant on times proportional to the size of the system $n$.

We do not show this property, yet extend the closeness to $M$ for exponential times, whereas in \cite{cf:BGP14} this is shown up to polynomial times.

\medskip

Theorem \ref{thm:sup}, as Theorem \ref{thm:finTim}, does not depend on the speed of convergence of the condition on the graph \eqref{h:graph}. The escaping time is indeed only due to the stochastic nature of the system, given by the Brownian motions, and it cannot be improved as explained above. The reason why one can control the perturbation induced by the graph structure for long times is somehow hidden in the martingale properties of $\mu^n_\cdot$ and in the fact that we do not really analyze the dynamics near $M$ (which can, a priori, depend on the graph). We refer to the proof of the subcritical regime for a clear control on the perturbations given by the graph, through the exponential stability of the stationary solution.

\subsection{Longtime behavior in the subcritical regime}
The subcritical regime is somehow easier than the supercritical regime since there is an unique stable stationary solution. We decide to include this case firstly because, to the author's knowledge, it is missing in the literature and, secondly, because the proof enlightens some aspects hidden in the supercritical regime. As a byproduct, we obtain the equivalent of maximal inequalities for Ornstein-Uhlenbeck processes in infinite dimensional Hilbert spaces, see Corollary \ref{cor:maxIne}.

\medskip

\begin{theorem}
	\label{thm:main} Let $0\leq K<1$. Suppose that condition \eqref{h:mu0} holds, i.e. for  all $\varepsilon_0 > 0$
	\begin{equation}
		\lim_{n\to \infty} \bP \left( \norm{\mu_0^n - \mu_0}_{-1} \leq \varepsilon_0 \right) = 1.
	\end{equation}
	Then, for every positive sequence $\{T_n\}_{n\in\N}$ such that $T_n =\exp(o(n))$, and for all $\varepsilon > 0$ small enough
	\begin{equation}
		\lim_{n\to \infty} \bP \left( \sup_{t\in[0, T_n]} \norm{\mu_t^n - \mu_t}_{-1} \leq \varepsilon \right) = 1.
	\end{equation}	
\end{theorem}

Observe that, since $\mu_\cdot$ converges as $t$ tends to infinity to $1/2\pi$ for all initial conditions $\mu_0$, then Theorem \ref{thm:main} implies the proximity of $\mu^n_t$ to the stable solution up to exponential times.

\medskip

Of independent interest, we present a corollary of Theorem \ref{thm:main} in the limit case $K=0$. This result seems to be well known, yet the author was unable to find it elsewhere. 

\medskip

\begin{cor}
	\label{cor:maxIne}
	Let $\mu^n_\cdot$ be the empirical measure of $n$ independent Brownian motions $\{ B^{j,n}_\cdot \}_{j=1,\dots,n}$ on $\bbT$ with initial conditions $\{\theta^i_0\}_{1\leq i \leq n}$ satisfying \eqref{h:mu0}. Then, there exist $C>0$ and $T_0>0$ such that for all $T>T_0$, the following maximal inequality holds:
	\begin{equation}
		\bE \left[ \sup_{t\in[T_0,T]} \norm{\mu^n_t - \frac 1{2\pi}}_{-1}^2 \right] \leq C \, \frac{\log(1+T-T_0)}n.
	\end{equation}
\end{cor}
Corollary \ref{cor:maxIne} shows a maximal inequality for the empirical measure of $n$ independent Brownian motions on the torus, establishing the SPDE version of the result for Ornstein-Uhlenbeck processes presented in \cite{cf:GP} for stochastic ordinary differential equations.

Observe that if the initial conditions and the graph are exchangeable (not necessarily independent), then Corollary \ref{cor:maxIne} and a classical result by Sznitman (\cite[Proposition 2.2]{cf:szni}) implies the creation of chaos for all times $T_n=o(\exp(n))$.

\medskip

\subsection{Organization of the paper}
This section ends presenting the existing literature and giving an outline of the proof for the three theorems.

Sections \ref{s:longTimeM}, \ref{s:longTimePi} and \ref{s:finTim} concern the proofs of the three results. In particular, Section \ref{s:longTimeM} is devoted to the long time dynamics close to $M$, it starts from the derivation of a mild formulation for the empirical measure, then proceeds with the control on the graph and the noise, and it ends with the proof of Theorem \ref{thm:sup}. Section \ref{s:longTimePi} concerns the subcritical regime where a different control on the perturbations is given. Finally, Section \ref{s:finTim} proves Theorem \ref{thm:finTim} by using slight variations of the previous techniques.

Appendix A gives a few examples of graph sequences that satisfy condition \eqref{h:graph}, together with remarks on the degrees and connectivity of such sequences. Appendix B contains information about the Hilbert spaces $H_{-1, \omega}$ and the linear operator $L_\psi$.

\medskip

\subsection{A glance at the existing literature}
\label{ss:lit}
The results presented here are at a crossroads of two different research areas: the long time dynamics of stochastic differential equations and the role of a network in a mean field model.

Concerning the long time behavior of weakly interacting particle systems, Theorems \ref{thm:sup} and \ref{thm:main} can be seen as a complement to the previous results presented in \cite{cf:BGP14}, filling the gap of the exponential time scale which has not been addressed so far. To the author's knowledge, they also represent the first equivalent, in infinite dimensional Hilbert spaces, to the famous result for stochastic ordinary differential equations in $\R^d$ by Friedlin and Wentzell (\cite{cf:FW}).

Looking at variations on the same model, the behavior of the classical Kuramoto model with intrinsic frequencies has been studied in \cite{cf:LP}, showing that the longtime dynamics is indeed dependent on the quenched setting given by the frequencies. A macroscopic constant speed in the phase appears on time scales of order $O(\sqrt{n})$, making the effects of the noise vanish. The results and the techniques presented here should be easily adaptable to this case showing the proximity to the manifold of solutions for long times, yet loosing the precise characterization of the motion on $M$ for which a deeper analysis is needed.

The Kuramoto model is an example of system which admits more than one stable stationary solution, a continuous manifold as already precised, and that's one of the reasons why it shows a rich phenomenology depending on the time scale under consideration. For similar results on different models, one has to dip in the context of stochastic partial differential equations (SPDEs)  with vanishing noise. Since the aim of this work is more oriented on the effects of the network rather than the longtime dynamics of SPDEs, the author refers to the bibliography in \cite{cf:BGP14,cf:LP} for a more comprehensive discussion.

\medskip

Turning to interacting particle systems on graphs, the subject has become an interesting topic in the mathematical community given the several applications to complex systems, in particular regarding the Kuramoto model and synchronization phenomena (e.g. \cite{cf:AKur,cf:RPJK16}), yet it has always been addressed on a finite time scale or up to times slowly diverging on $n$, i.e. $T_n = O (\log n)$.

The first articles \cite{cf:BBW,cf:DGL} attack the problem under a propagation of chaos viewpoint, requiring independent and identically distributed initial conditions and also independent of the realization of the graph. In this setting, the condition on the graph boils down to a condition on the degrees only so that very general graphs are allowed (see again \cite[Remark 1.4]{cf:DGL} and \cite[Remark 1.2]{cf:CDG}). Regarding more inhomogeneous settings, \cite{cf:L18} extends \cite{cf:DGL} to graphons and \cite{cf:RO} presents a Large Deviation result again in the graphon setting. Observe that \cite{cf:RO} already makes use of Grothendieck's inequality and the norm $\norm{\cdot}_{\infty \to 1}$ to control the graphs fluctuations. Up to now, the only result not assuming independence in the initial data is given by \cite{cf:CDG}, where general systems of interacting particles are defined on Erd\H{o}sh-Rényi random graphs and the empirical measure is shown to satisfy a Law of Large Number and a Large Deviation principle, implying the convergence to the respective mean field limit.

In the deterministic setting, the Kuramoto model has been studied on different networks with various hypothesis on the initial conditions, we refer to \cite{cf:CM16, cf:Med19} and references therein.

In all the cited works, the condition on the normalization $p_n$ is slightly stronger, or equivalent, to the one required in \eqref{h:graph}, see in particular Propositions \ref{p:er} and \ref{p:ram} in Appendix A. If $np_n$ is not diverging as $n$ tends to infinity, i.e. the case of sparse graphs, the limiting behavior of the empirical measure seems to be rather different from the mean field limit, see \cite{cf:LRW, cf:ORS}.

\medskip

\subsection{Outline of the proofs}
The three theorems are proven in a similar way and the main ingredients are given by
\begin{enumerate}
	\item A mild formulation satisfied by $\mu^n_\cdot$ for each $n\in \N$;
	\item The control on the perturbations given by the graph structure through Grothendieck's Inequality;
	\item The control on the random perturbations given by the Brownian motions through maximal inequalities for self-normalized processes.
\end{enumerate}
The mild formulation will be different in all the three cases and will depend on the linear dynamics around $M$ or $1/2\pi$ or on the properties of $\mu_t$. In Section \ref{s:longTimeM}, we give a full derivation of the stochastic partial differential equation satisfied by $\mu^n_t - q_\psi$ in a neighborhood of $M$.

The control on the graph will also depend whether there is a strong contraction given by the dynamics, or not. Whenever the evolution  is contracting in all direction, as in the subcritical case around $1/2\pi$, these perturbations can be controlled uniformly in time, we refer to Lemma \ref{lem:g^n_tPi} for a precise statement.

A fine control on the random perturbations turns out to be rather delicate and one has to exploit all the properties associated to the Hilbert structure as well as the ones associated to the linear dynamics around the stationary solutions to get the job done. We give two independent explicit proofs:
\begin{itemize}
	\item Around $M$, we study the noise using of a strong result on self-normalized martingales (\cite{cf:dlPKL04});
	\item Around $1/2\pi$, we extend a result on maximal inequalities for Ornstein-Uhlenbeck processes (\cite{cf:GP}), to an infinite dimensional setting.
\end{itemize}

\medskip

Once the perturbations are controlled, a Gronwall-like lemma is used to bound the difference between $\mu^n_\cdot$ and the relative target. In the supercritical case, we need to set up an (easy) iterative scheme in order to estimate the distance between $\mu^n_\cdot$ and $M$ on bounded time intervals, and then make use of the martingale property of system $\eqref{eq:k}$ to extend the result up to almost exponential times; in the subcritical case, the result is directly obtained by the bound on the noise, the graph perturbations are indeed controlled for all times.

\medskip

\section{Longtime dynamics close to $M$}
\label{s:longTimeM}
In a neighborhood of $M$, one can exploit the properties of the linear dynamics around $q_\psi$. For $u \in \cL^2_0$, let
\begin{equation}
\label{def:linOp_q}
	L_\psi u  \; := \; \tfrac 12 \partial^2_\theta u - \partial_\theta\left[u(J*q_\psi) + q_\psi (J*u)\right],
\end{equation}
be the linear operator at $q_\psi$, its domain is given by $D(L_\psi) = \{ u \in \cC^2(\bbT), \; \int_{\bbT} u(\theta) \dd \theta = 0\}$. The operator $L_\psi$ is self-adjoint in $H_{-1,1/q_\psi}$ and its adjoint $L^*_\psi$ in $\cL^2_0$ has the following expression
\begin{equation}
\label{def:linOp_q*}
L^*_\psi u  \; = \; \tfrac 12 \partial^2_\theta u + (J*q_\psi) \partial_\theta u - (J*  q_\psi \, \partial_\theta u) - \int_{\bbT}  (J* q_\psi \,\partial_\theta u),
\end{equation}
and domain $D(L^*_\psi)=D(L_\psi)$.

We recall here the most important properties of $L_\psi$, referring to Appendix B for more informations. The linear operator $-L_\psi$ has compact resolvent and its spectrum lies in $[0,\infty)$: the smallest eigenvalue  $\lambda_0 :=0$, associated to the eigenfunction $\partial_\theta q_\psi$, is isolated from the rest of the spectrum. In particular, this implies that $H_{-1}$ can be decomposed into a direct sum $T_\psi \oplus N_\psi$, where $T_\psi = \textup{Span} (\partial_\theta q_\psi)$; we denote by $P^0_\psi$ the projection on $T_\psi$ along $N_\psi$ and $P^s_\psi = 1- P^0_\psi$. Observe that both $P^0_\psi$ and $P^s_\psi$ commute with $L_\psi$ (e.g. \cite{cf:henry}).

Let $\{\lambda_0 <\lambda_1 \leq \dots\}\subset [0,\infty)$ denote the set of eigenvalues and let $\{e^\psi_l\}_{l=0,1,\dots}$ be the correspondent set of eigenfunctions, normalized in $H_{-1,1/q_\psi}$, i.e.
\begin{equation}
\label{d:eig}
	-L_\psi e^\psi_l = \lambda_l e^\psi_l, \quad\text{ for } l=0,1,\dots.
\end{equation}
Observe that $e^\psi_l \in \cC^\infty(\bbT)$. Moreover, the eigenvalues do not depend on the phase $\psi$, whereas the eigenfunctions do in a rather simple way given by the rotation symmetry of the system, i.e.
\begin{equation}
	e^\psi_l (\cdot) = e^\theta_l (\cdot + \psi - \theta), \quad \text{ for } l=0,1,\dots.
\end{equation}
As a matter of fact, we will study the system only around some $q_\psi$. The dual eigenfunctions $f^\psi_l$ associated to $L^*_\psi$ will play an important role for studying the noise perturbation, their properties are studied in Proposition \ref{p:eigFun}.

From the previous properties one deduces that $L_\psi$ (respectively $L^*_\psi$) generates a strong continuous semigroup on $H_{-1}$ (resp. $H_1$) that we denote by $e^{tL_\psi}$ (resp. $e^{tL^*_\psi}$). These semigroups have many properties that we will recall and use throughout this section, see Proposition \ref{p:linSem} for a general statement.

\medskip

A final remark: we use the letter $C$ for all the constants even if they are possibly different, the value of $C$ can change from one line to another if the constant is replaced by another constant and the context is clear.


\subsection{The mild formulation around $q_\psi \in M$}
As shown in \cite{cf:BGP14}, $\bar{\mu}^n_\cdot$ satisfies a SPDE written in mild form once it is close to $M$; in this subsection we extend this formulation to $\mu^n_t$. Let $\nu^n_\cdot:= \mu^n_\cdot - q_\psi$, then

\medskip

\begin{proposition}
	\label{p:mildL}
	The process $\nu^n_t \in H_{-1}$ satisfies the following stochastic partial differential equation in $\cC^0\left([0,T], H_{-1} \right)$:
	\begin{equation}
	\label{eq:lMild2}
	\nu^n_t = e^{t L_\psi}\nu^n_0 - \int_{0}^{t} e^{(t-s)L_\psi} \partial_\theta \left[\nu^n_s (J*\nu^n_s)\right] \dd s - g^n_t + z^n_t,
	\end{equation}
	where
	\begin{equation}
	g^n_t = \frac 1{n^2} \sum_{i,j=1}^{n} \int_0^t \left(\frac{\xi_{ij}}{p_n} - 1\right)  e^{(t-s)L_\psi} \partial_\theta \left[\delta_{\theta^{i,n}_s} ( J*\delta_{\theta^{j,n}_s})\right] \dd s,
	\end{equation}
	and $z^n_t \in H_{-1}$ is defined for $h \in H_1$ by
	\begin{equation}
	\langle z^n_t, h \rangle_{-1,1} = \frac 1n \sum_{j=1}^n \int_0^t \left[\partial_\theta e^{(t-s)L^*_\psi}  h\right] (\theta^{j,n}_s) \dd B^j_s.
	\end{equation}
\end{proposition}

\begin{proof}
	Let $F=F_t(\theta) \in \cC^{1,2} \left([0, \infty)\times\bbT \right)$, with $\int F_t = 0$ for all $t\geq 0$. For some $t \geq 0$, an application of Ito formula, together with the definition of $L^*_\psi$ gives
	\begin{equation}
	\begin{split}
	\label{eq:lMild0}
	\langle \mu^n_t - q_\psi , F_t \rangle & = \langle \mu^n_0 - q_\psi , F_0 \rangle +  \int_0^t \langle \mu^n_s - q_\psi , \partial_s F_s + L^*_\psi F_s \rangle \dd s + \\
	& + \int_0^t \langle (\mu^n_s - q_\psi) (J*(\mu^n_s - q_\psi)),  \partial_\theta F_s \rangle \dd s + G^n_t (F) + Z^n_t(F),
	\end{split}
	\end{equation}
	where
	\begin{equation}
	\begin{split}
	G^n_t (F) & = \frac 1{n^2} \sum_{i,j=1}^{n} \int_0^t \left(\frac{\xi_{ij}}{p_n} - 1\right) J(\theta^{i,n}_s- \theta^{j,n}_s) \partial_\theta F_s (\theta^{i,n}_s) \dd s,\\
	Z^n_t (F) & = \frac 1n \sum_{j=1}^n \int_0^t \partial_\theta F_s (\theta^{j,n}_s) \dd B^j_s.
	\end{split}
	\end{equation}
	The properties of $e^{t L^*_\psi}$, see Proposition \ref{p:linSem}, assure that the function
	\begin{equation}
	F = F_s (\theta) = e^{(t-s)L^*_\psi} h(\theta), \quad \text{ for some } h\in \cC^2(\bbT), \; \int h = 0,
	\end{equation}
	is $\cC^{1,2} ([0,t]\times \bbT)$. But then $\partial_s F_s = - L^*_\psi F_s$ and one obtains
	\begin{equation}
	\label{eq:lMild1}
	\langle \nu^n_t , F_t \rangle  = \langle \nu^n_0,  e^{tL^*_\psi} h \rangle
	+ \int_0^t \langle \nu^n_s (J*\nu^n_s),  \partial_\theta e^{(t-s)L^*_\psi} h \rangle \dd s + g^n_t (h) + z^n_t(h),
	\end{equation}
	where we have used the definition of $\nu^n_t$ and the notations
	\begin{align}
	\label{eq:lg_n^t}
	g^n_t (h) & = \frac 1{n^2} \sum_{i,j=1}^{n} \int_0^t \left(\frac{\xi_{ij}}{p_n} - 1\right) J(\theta^{i,n}_s- \theta^{j,n}_s)  \left[ \partial_\theta e^{(t-s)L^*_\psi} h \right]  (\theta^{i,n}_s) \dd s,\\
	z^n_t (h) & = \frac 1n \sum_{j=1}^n \int_0^t \left[\partial_\theta e^{(t-s)L^*_\psi} h\right]  (\theta^{j,n}_s) \dd B^j_s.
	\end{align}
	We aim at proving that \eqref{eq:lMild1} is the weak formulation of the mild equation \eqref{eq:lMild2} in $H_{-1}$.
	
	\medskip
	
	Let $\{\nu_l\}_{l\geq 1} \subset \cL^2_0$ such that $\nu_l \xrightarrow{l\uparrow\infty} \nu^n_0$ in $H_{-1}$. Then, for $h \in \cC^2$
	\begin{equation}
	\label{b:nu_l}
	\langle \nu_l , e^{t L^*_\psi} h \rangle_{-1,1} = 	\left( \nu_l , e^{t L^*_\psi} h \right) =  \left(e^{t L_\psi} \nu_l ,  h\right)   = 	\langle e^{t L_\psi} \nu_l ,  h \rangle_{-1,1}.
	\end{equation}
	By continuity of the operators, $e^{t L_\psi}  \nu_l$ converges in $H_{-1}$ to $e^{t L_\psi} \nu^n_0$ as $l\uparrow\infty$. Taking the limit for $l\uparrow \infty$ in both sides of \eqref{b:nu_l}, we deduce
	\begin{equation}
	\langle \nu^n_0 , e^{t L^*_\psi} h \rangle_{-1,1} = \langle e^{t L_\psi} \nu^n_0 ,  h \rangle_{-1,1}.
	\end{equation}
	
	We now focus on
	\begin{equation}
	\omega^n_s := \nu^n_s (J*\nu^n_s).
	\end{equation}
	Consider $\{\nu_{s,l}\}_{l \geq 1} \subset \cL^2_0$ which converges to $\nu^n_s$ in $H_{-1}$ as $l\uparrow\infty$, and define
	\begin{equation}
	\omega_{s,l} := \nu_{s,l} (J*\nu^n_s).
	\end{equation}
	For any $l\geq 1$, it holds
	\begin{equation}
	\begin{split}
	& \langle \omega_{s,l} , \partial_\theta e^{(t-s) L^*_\psi} h \rangle_{-1,1} = \left( \omega_{s,l}, \partial_\theta e^{(t-s) L^*_\psi}  h  \right) = \\
	& = - \left( e^{(t-s) L_\psi} \partial_\theta \, \omega_{s,l},  h  \right) = - \langle e^{(t-s) L_\psi} \partial_\theta \, \omega_{s,l},  h  \rangle_{-1,1}.
	\end{split}
	\end{equation}
	Using the properties of the semigroup one obtains
\begin{equation}
	\begin{split}
	\left| \langle e^{(t-s) L_\psi} \partial_\theta (\omega_{s,l}-\omega^n_s) , h \rangle_{-1,1} \right|  \leq \norm{h}_1 \norm{e^{(t-s) L_\psi} \partial_\theta (\omega_{s,l} - \omega^n_s)}_{-1} \leq \\
	\leq  \norm{h}_1 \frac{C}{\sqrt{t-s}} \norm{\partial_\theta (\omega_{s,l} - \omega^n_s)}_{-2} = \norm{h}_1 \frac{C}{\sqrt{t-s}} \norm{\omega_{s,l} - \omega^n_s}_{-1},
	\end{split}
\end{equation}
	which implies
	\begin{equation}
	\label{b:contHeat}
	\norm{e^{(t-s) L_\psi} \partial_\theta (\omega_{s,l}-\omega^n_s)}_{-1} \leq \frac{C}{\sqrt{t-s}} \norm{\omega_{s,l}-\omega^n_s}_{-1}.
	\end{equation}
	Since $h$ is regular and $\omega_{s,l} \xrightarrow{l\uparrow\infty} \omega^n_s$ in $H_{-1}$, this implies
	\begin{equation}
	\langle \omega^n_s , \partial_\theta e^{(t-s) L^*_\psi} h \rangle_{-1,1} = - \langle e^{(t-s) L_\psi} \partial_\theta \, \omega^n_s , h \rangle_{-1,1}.
	\end{equation}
	
	We now observe from \eqref{b:contHeat} that
	\begin{equation}
	\norm{e^{(t-s) L_\psi} \partial_\theta \omega^n_s}_{-1} \leq \frac{C}{\sqrt{t-s}}
	\end{equation}
	thus the integral in \eqref{eq:lMild2}
	\begin{equation}
	\label{eq:stocInt}
	\int_{0}^{t} e^{(t-s)L_\psi} \partial_\theta \left[ \nu^n_s (J*\nu^n_s)\right]\dd s
	\end{equation}
	is almost surely finite. Using \cite[Theorem 1, p.133]{cf:YOS}, we deduce that  \eqref{eq:stocInt} makes sense as a Bochner integral in $H_{-1}$. The continuity is a direct consequence of the continuity of $e^{t L_\psi}$. 
	
	\medskip
	
	Assume that $g^n_t(h)= \langle g^n_t, h \rangle_{-1,1}$ and $z^n_t(h) = \langle z^n_t, h\rangle_{-1,1}$ are well defined and continuous with respect to $t$ for all $h\in H_1$; we have shown that
	\begin{equation}
	\label{eq:lMild3}
	\begin{split}
	\langle  \nu^n_t  ,   h \rangle_{-1,1}& =  \langle e^{t L_\psi} \nu^n_0 , h \rangle_{-1,1} + \\
	& - \langle \int_0^t e^{(t-s) L_\psi} \partial_\theta  \left[ \nu^n_s(J*\nu^n_s) \right] \dd s,  h \rangle_{-1,1} - \langle g^n_t, h \rangle_{-1,1} + \langle z^n_t, h \rangle_{-1,1}.
	\end{split}
	\end{equation}
	Since \eqref{eq:lMild3} holds for all $h \in H_1$, the identity \eqref{eq:lMild2} follows. All elements in \eqref{eq:lMild2} take values in $\cC^0([0,T], H_{-1})$ and the proof is then concluded modulo regularity and wellposedness of $g^n_\cdot$ and $z^n_\cdot$. We refer to Lemma \ref{lem:g^n_t} and Lemma \ref{lem:z^n_t} which are presented in the next subsection.
\end{proof}

\subsection{Control on the perturbations} Two kinds of perturbations are present in the SPDE \eqref{eq:lMild2}: $z^n_\cdot$ given by the stochastic nature of the system and $g^n_\cdot$ given by the presence of a network structure. In this subsection, we exhibit the control over the two perturbations. Observe that all the estimates are independent of $\psi$.

We start with the control on the graph structure, which uses Grothendieck's Inequality presented in Theorem \ref{thm:gro}.

\medskip

\begin{lemma}[Wellposedness and bounds on $g^n_t$]
	\label{lem:g^n_t}
	For $n\in\N$ and $t\geq 0$, let $g^n_t$ be given by
	\begin{equation}
	g^n_t = \frac 1{n^2} \sum_{i,j=1}^{n} \int_0^t \left(\frac{\xi_{ij}}{p_n} - 1\right)  e^{(t-s)L_\psi} \partial_\theta \left[\delta_{\theta^{i,n}_s} ( J*\delta_{\theta^{j,n}_s})\right] \dd s.
	\end{equation}
	Then
	\begin{enumerate}
		\item $g^n \in \cC^0 ([0,\infty), H_{-1})$. In particular, for every $h\in H_1$ and $t \geq 0$
		\begin{equation}
		\label{g:equiv}
		g^n_t (h) = - \frac 1{n^2} \sum_{i,j=1}^{n} \int_0^t \left(\frac{\xi_{ij}}{p_n} - 1\right) J(\theta^{i,n}_s- \theta^{j,n}_s)  \left[\partial_\theta e^{(t-s)L^*_\psi}  h\right] (\theta^{i,n}_s) \dd s.
		\end{equation}
		\item There exists $D>0$ such that
		\begin{equation}
		\label{g:supT}
		\norm{g^n_t}_{-1} \leq D \sqrt{t} \, \frac{\gnorm{P^{(n)}-\mathbf{1}^{(n)}}}{n^2}, \quad \text{ for all } t\geq 0.
		\end{equation}
	\end{enumerate}
\end{lemma}

\begin{proof}
	Fix $n$ large. Consider $\{\phi_l\}_{l \geq 1} \subset \cC^\infty$ such that $\phi_l \geq 0,\; \phi_l (\theta) = 0$ for $\theta \in [1/l, 2\pi - 1/l], \; \int \phi_l =1$ for every $l \geq 1$ and $\lim_{l\to \infty} \int F \phi_l = F(0)$ for every $F \in \cC^0$. For $i=1,\dots, n$, define
	\begin{equation}
	\phi^{i}_{s,l} := \phi_l * \delta_{\theta^{i,n}_s}.
	\end{equation}
	We start by establishing \eqref{g:equiv}. For each $h\in \cC^2$
\begin{equation}
	\begin{split}
	& \langle \frac 1{n^2} \sum_{i,j=1}^{n} \left(\frac{\xi_{ij}}{p_n} - 1\right)  \phi^{i}_{s,l} ( J*\delta_{\theta^{j,n}_s}), \partial_\theta e^{(t-s)L^*_\psi} h \rangle_{-1,1} = \\
	& = \left( \frac 1{n^2} \sum_{i,j=1}^{n} \left(\frac{\xi_{ij}}{p_n} - 1\right) \phi^{i}_{s,l} ( J*\delta_{\theta^{j,n}_s}), \partial_\theta e^{(t-s)L^*_\psi} h \right) = \\
	& = - \left( \frac 1{n^2} \sum_{i,j=1}^{n} \left(\frac{\xi_{ij}}{p_n} - 1\right) e^{(t-s)L_\psi} \partial_\theta \left[ \phi^{i}_{s,l} ( J*\delta_{\theta^{j,n}_s})\right],   h \right) = \\
	& = - \langle \frac 1{n^2} \sum_{i,j=1}^{n} \left(\frac{\xi_{ij}}{p_n} - 1\right) e^{(t-s)L_\psi} \partial_\theta \left[ \phi^{i}_{s,l} ( J*\delta_{\theta^{j,n}_s})\right],   h \rangle_{-1,1}
	\end{split}
\end{equation}
	But $\frac 1{n^2} \sum_{i,j=1}^{n} \left(\frac{\xi_{ij}}{p_n} - 1\right)  \phi^{i}_{s,l} ( J*\delta_{\theta^{j,n}_s})$ converges to $\frac 1{n^2} \sum_{i,j=1}^{n} \left(\frac{\xi_{ij}}{p_n} - 1\right) \delta_{\theta^{i,n}_s} ( J*\delta_{\theta^{j,n}_s})$ since
	\begin{equation}
	\norm{\frac 1{n^2} \sum_{i,j=1}^{n} \left(\frac{\xi_{ij}}{p_n} - 1\right)  \left(\phi^{i}_{s,l} - \delta_{\theta^{i,n}_s} \right) ( J*\delta_{\theta^{j,n}_s})}_{-1} \leq \frac 1{p_n} \sup_{i=1,\dots,n} \norm{\phi^{i}_{s,l} - \delta_{\theta^{i,n}_s}}_{-1},
	\end{equation}
	which tends to zero as $l$ tends to infinity.
	
	Thanks to the properties of the semigroup, the same holds true for 
	\begin{equation}
		\frac 1{n^2} \sum_{i,j=1}^{n} \left(\frac{\xi_{ij}}{p_n} - 1\right) e^{(t-s)L_\psi} \partial_\theta \left[ \left(\phi^{i}_{s,l} - \delta_{\theta^{i,n}_s} \right) ( J*\delta_{\theta^{j,n}_s}) \right];
	\end{equation}
	indeed, by Proposition \ref{p:linSem}, 
	\begin{equation}
	\begin{split}
	\norm{\frac 1{n^2} \sum_{i,j=1}^{n} \left(\frac{\xi_{ij}}{p_n} - 1\right) e^{(t-s)L_\psi} \partial_\theta \left[ \left(\phi^{i}_{s,l} - \delta_{\theta^{i,n}_s} \right) ( J*\delta_{\theta^{j,n}_s}) \right]}_{-1} \leq \\
	\leq \frac{C}{p_n \sqrt{t-s}} \sup_{i=1,\dots,n} \norm{\phi^{i}_{s,l} - \delta_{\theta^{i,n}_s}}_{-1}.
	\end{split}
	\end{equation}
	A similar argument shows that
	\begin{equation}
	\norm{\frac 1{n^2} \sum_{i,j=1}^{n} \left(\frac{\xi_{ij}}{p_n} - 1\right) e^{(t-s)L_\psi} \partial_\theta \left[\delta_{\theta^{i,n}_s} ( J*\delta_{\theta^{j,n}_s}) \right]}_{-1} \leq  \frac{C}{p_n \sqrt{t-s}},
	\end{equation}
	which, in turn, implies that
	\begin{equation}
	\frac 1{n^2} \sum_{i,j=1}^{n} \int_0^t \left(\frac{\xi_{ij}}{p_n} - 1\right)  e^{(t-s)L_\psi} \partial_\theta \left[\delta_{\theta^{i,n}_s} ( J*\delta_{\theta^{j,n}_s})\right] \dd s
	\end{equation}
	is almost surely finite and continuous with respect to $t$. We deduce \eqref{g:equiv}.
	
	\medskip
	
	For the second part \eqref{g:supT}, observe that
	\begin{equation}
	\begin{split}
	\label{g:L2sym}
	\langle \frac 1{n^2} \sum_{i,j=1}^{n} \left(\frac{\xi_{ij}}{p_n} - 1\right) e^{(t-s)L_\psi} \partial_\theta \left[ \delta_{\theta^{i,n}_s} ( J*\delta_{\theta^{j,n}_s}) \right],   h \rangle_{-1,1} = \\
	= - \frac 1{n^2} \sum_{i,j=1}^{n} \left(\frac{\xi_{ij}}{p_n} - 1\right)   \langle \delta_{\theta^{i,n}_s},   ( J*\delta_{\theta^{j,n}_s}) \partial_\theta e^{(t-s)L^*_\psi} h \rangle_{-1,1}.
	\end{split}
	\end{equation}
	We claim that this last term can be controlled by $\gnorm{P^{(n)}- \mathbf{1}^{(n)}}$ through Grothendieck's inequality. By choosing $H=H_{-1}$ and
	\begin{equation}
	\begin{split}
	&a_{ij} = \left(\tfrac{\xi_{ij}}{p_n} - 1\right), \\
	&S_i = \delta_{\theta^{i,n}_s}, \\
	&T_j = \frac{\sqrt{t-s}}{C}\left(J*\delta_{\theta^{j,n}_s}\right) \partial_\theta e^{(t-s) L^*_\psi} \, \frac h{\;\norm{h}_1},
	\end{split}
	\end{equation}
	Theorem \ref{thm:gro} allows us to bound the expression in \eqref{g:L2sym} by 
	\begin{equation}
	K_R \frac{C}{\sqrt{t-s}} \norm{h}_1  \; \frac{\gnorm{P^{(n)}-\mathbf{1}^{(n)}}}{n^2}.
	\end{equation}
	This shows that
	\begin{equation}
	\norm{g^n_t}_{-1} \leq K_R C \frac{\gnorm{P^{(n)}-\mathbf{1}^{(n)}}}{n^2} \int_0^t \frac 1{\sqrt{t-s}}  \dd s \; = D \sqrt{t} \frac{\gnorm{P^{(n)}-\mathbf{1}^{(n)}}}{n^2} \; ,
	\end{equation}
	where $D :=  K_R C / 2 > 0$.	The proof is concluded.
\end{proof}

\medskip

We now turn to the stochastic term $z^n_\cdot$ in \eqref{eq:lMild2}. Recall that $L^*_\psi$ is diagonal in the basis $\{f^\psi_l\}_{l\geq 0}$ of $H_{1,q_\psi}$, with eigenvalues denoted by $\{\lambda_l\}_{\lambda\geq 0}$, see Proposition \ref{p:eigFun}. We precisely analyze $z^n_\cdot$ through its coefficients in the orthonormal basis given by $L^*_\psi$.

\medskip

\begin{lemma}[Wellposedness and bounds on $z^n_t$]
	\label{lem:z^n_t}
	For $n\in\N$ and $t > 0$, let $z^n_t$ be defined by
	\begin{equation}
	z^n_t = \sum_{l\geq 1}  \langle z^n_t, e^\psi_l \rangle_{H_{-1,1/q_\psi}} \; e^\psi_l.
	\end{equation}
	Then
	\begin{enumerate}
		\item $z^n_\cdot \in \cC^0 ([0,\infty), H_{-1})$ almost surely. In particular, for every $h \in H_1$
		\begin{equation}
		z^n_t (h) = \frac 1n \sum_{j=1}^n \int_0^t \left[\partial_\theta e^{(t-s)L^*_\psi} h\right]  (\theta^{j,n}_s) \dd B^j_s.
		\end{equation}
		\item For every $T>0$, there exists a constant $Z=Z(T) > 0$, such that for $n$ large enough it holds that
		\begin{equation}
		\label{z:longTime}
		\forall \eta >0, \quad \bP \left( \sup_{t\in[0,T]} \norm{z^n_t}_{-1} > \eta \right) \leq \exp \left\{-Z n \eta^2 \right\}.
		\end{equation}
	\end{enumerate}
\end{lemma}

\begin{proof} We start by observing that the definition of $z^n_t$ in the basis of $H_{-1,1/q_\psi}$, coincides with the one give in the mild formulation. For $h = \sum_{l\geq 0} \langle h, f^\psi_l \rangle_{H_{1,q_\psi}} f^\psi_l$, one obtains
	\begin{equation}
	\begin{split}
		\langle z^n_t, h \rangle_{-1,1} &= \sum_{l\geq 0}  \langle z^n_t, e^\psi_l \rangle_{H_{-1,1/q_\psi}} \, \langle h, f^\psi_l \rangle_{H_{1,q_\psi}} =\\
		& = \sum_{l\geq 0} z^n_t (f^\psi_l) \langle h, f^\psi_l \rangle_{H_{1,q_\psi}} = z^n_t (h),
	\end{split}
	\end{equation}
where we have used the properties of $e^\psi_l$ and $f^\psi_l$, see Proposition \ref{p:eigFun}.

	Before proving (1), we prove (2) and this will imply the existence of a continuous version of $z^n_\cdot$ almost surely.
	
	Concerning (2), we start by observing that
	\begin{equation}
		\norm{z^n_t}_{-1, 1/q_\psi} = \sum_{l\geq 0} \left| z^n_t(f^\psi_l) \right|^2.
	\end{equation}
	Let $l \geq 1$ and consider $z^n_t(f^\psi_l)$, by the definition of $f^\psi_l$ and the properties of the semigroup, one gets
	\begin{equation}
		z^n_t(f^\psi_l) =  \frac 1n \sum_{j=1}^n \int_0^t  \left[ \partial_\theta e^{(t-s)L^*_\psi} f^\psi_l \right] (\theta^{j,n}_s) \dd B^j_s = \frac 1n \sum_{j=1}^n \int_0^t e^{-(t-s)\lambda_l} \left[ \partial_\theta f^\psi_l \right] (\theta^{j,n}_s) \dd B^j_s.
	\end{equation}
	Set $c = \sup_{l\geq 0} \left|\partial_\theta f^\psi_l\right| < \infty$, see Proposition \ref{p:eigFun}. We rewrite the last expression as
	\begin{equation}
	\label{z:A_mar}
		z^n_t(f^\psi_l) = \frac{c e^{-t\lambda_l}}{\sqrt{2 \lambda_l n}} \,  A_t,
	\end{equation}
	where $A_t$ is a continuous martingale given by
	\begin{equation}
		A_t = \sqrt{\frac{2\lambda_l}{c^2 n}} \sum_{j=1}^n \int_0^t e^{s \lambda_l} \left[\partial_\theta f^\psi_l \right] (\theta^{j,n}_s) \dd B^j_s 
	\end{equation}
	and quadratic variation bounded by
	\begin{equation}
	\label{z:A_quadVar}
		\langle A \rangle_t = \frac{2\lambda_l}{c^2 n} \sum_{j=1}^n \int_0^t e^{2s\lambda_l} \left[\partial_\theta f^\psi_l \right]^2 (\theta^{j,n}_s) \dd s \leq 2\lambda_l \int_0^t e^{2\lambda_l s} \dd s \leq e^{2t\lambda_l} -1.
	\end{equation}
	From \eqref{z:A_mar} and \eqref{z:A_quadVar}, one deduces that $z^n_t (f^\psi_l)$ is a self normalized process. For estimating $\bP \left( \sup_{t \in [0,T]} \left| z^n_t(f^\psi_l)\right|^2 > \eta \right)$, we can this use the following result
	\begin{theorem}[{\cite[Theorem 4.1 and the following remark]{cf:dlPKL04}}]
		\label{thm:dlPKL04}
		Let $T>0$, $\alpha \in (0, \frac 12)$ and $\left(A_t\right)_{t \in [0,T]}$ be a martingale with $A_0=0$. There exists $C>0$, depending only on $\alpha$, such that
		\begin{equation}
			\bE \left[ \sup_{t\in[0,T]} \exp \left\{ \frac{ \alpha A^2_t}{\langle A \rangle_t \log \log \left(\langle A \rangle_t \vee e^2\right)} \right\}\right] \leq C.
		\end{equation}		
	\end{theorem}
	
	By standard computations, we obtain
	\begin{equation}
	\begin{split}
		&\bP \left( \sup_{t \in [0,T]} \left| z^n_t(f^\psi_l)\right|^2 > \eta \right) = \bP \left( \sup_{t \in [0,T]} e^{-2t \lambda_l} A^2_t > \frac{2 \lambda_l n}{c^2} \eta \right) = \\
		&= \bP \left( \sup_{t \in [0,T]} \frac{A^2_t}{4 e^{2t \lambda_l} \log (2 + 2T \lambda_l)} > \frac{\lambda_l n}{2c^2 \log(2+2T\lambda_l)} \eta  \right) \leq\\
		& \leq \bE \left[ \sup_{t \in [0,T]} \exp \left\{  \frac{A^2_t}{4 e^{2t \lambda_l} \log (2 + 2T \lambda_l)}\right\} \right] \exp \left\{- \frac{\lambda_l n}{2c^2 \log(2+2T\lambda_l)} \eta\right\}.
	\end{split}
	\end{equation}
	We now use the fact that
	\begin{equation}
		\alpha := \sup_{t\in[0,T]} \frac{\langle A \rangle _t \log \log (\langle A \rangle_t \vee e^2)}{4 e^{2t \lambda_l} \log (2 + 2T \lambda_l)} \leq \frac 14, \quad \text{for all } l \geq 1,
	\end{equation}
	and, by Theorem \ref{thm:dlPKL04}, we obtain that there exists $C>0$, independent of $T,n$ and $l$, such that
	\begin{equation}
	\label{z:expIn_l}
		\bP \left( \sup_{t \in [0,T]} \left| z^n_t(f^\psi_l)\right|^2 > \eta \right) \leq C \exp \left\{ - \frac{\lambda_l n}{c^2 \log(2+2T\lambda_l)} \eta  \right\}.
	\end{equation}
	
	The case $l=0$ is somehow easier since
	\begin{equation}
		z^n_t(f^\psi_0) = \frac 1n \sum_{j=1}^n \int_0^t \left[ \partial_\theta f^\psi_0 \right] (\theta^{j,n}_s) \dd B^j_s
	\end{equation}
	is a standard martingale with bounded quadratic variation: one can use Theorem \ref{thm:dlPKL04} or, more simply, exponential estimates and Doob's inequality to obtain that there exists $c_0 >0$ (depending on $T$) such that for all $\eta >0$
	\begin{equation}
	\label{z:expIn_0}
		\quad \bP \left(\sup_{t \in[0,T]} \left|z^n_t(f^\psi_0)\right|^2 > \eta \right) \leq c_0 \exp\{-c_0 n \eta\}.
	\end{equation}
	
	\medskip
	
	The last part of the proof consists in exploiting the exponential inequalities \eqref{z:expIn_l} and \eqref{z:expIn_0}, and to transfer them to $\sup_{t \in [0,T]} \norm{z^n_t}_{-1}^2$. For this purpose, let $S>0$ be defined by
	\begin{equation}
		S := \left(\sum_{l\geq 0} \frac 1{(1+l)^{4/3}}\right)^{-1}.
	\end{equation}
	For some $\eta >0 $, it holds that
	\begin{equation}
	\begin{split}
		\bP  &\left(\sup_{t\in[0,T]} \norm{z^n_t}_{-1,1/q_\psi} > \eta \right) \leq \bP \left( \sum_{l\geq 0} \sup_{t\in[0,T]} \left| z^n_t (f^\psi_l)\right|^2 > \eta^2  \right) \leq \\
		&\leq \sum_{l\geq 0} \bP \left( \sup_{t\in[0,T]} \left| z^n_t (f^\psi_l)\right|^2 > \frac{S}{(1+l)^{4/3}} \eta^2  \right) \leq  \\
		&\leq c_0 \exp \{-c_0 S \, n \eta^2\} + C \sum_{l\geq 1} \exp \left\{ - \frac{S}{c^2} \, \frac{\lambda_l}{\log(2+2T\lambda_l) (1+l)^{4/3}} \, n \eta^2 \right\}.
	\end{split}
	\end{equation}
	Now we use the fact that $\lambda_l = \Theta(l^2)$ as $l$ tends to infinity, see Proposition \ref{p:eigFun}. In particular, there exists $L>0$, depending on $T$, such that
	\begin{equation}
	\begin{split}
		\bP  \left(\sup_{t\in[0,T]} \norm{z^n_t}_{-1,1/q_\psi} > \eta \right)&  \leq  c_0 \exp \{-c_0 S \, n \eta^2\} +\\
		 + C L \exp & \left\{ - \frac{S}{c^2 \log(2+2T\lambda_L)} n \eta^2 \right\}  + C \sum_{l > L} \exp \left\{ - \frac S{c^2} \sqrt{l} \,  n \eta^2  \right\}.
	\end{split}
	\end{equation}
	Observing that $\int_1^\infty e^{-\sqrt{x}n} \dd x = \frac{2(n+1)}{n^2} e^{-n}$, taking $n$ large enough and $Z$ an suitable constant depending on $c_0, S, L$ and $C$, the proof of (2) is concluded.
	
	\medskip
	
	Back to (1), observe that for $s,t \in [0,T]$ and for some $k\geq 1$
	\begin{equation}
	\label{z:cont}
	\norm{z^n_t - z^n_s}_{-1}^2 \leq \sum_{l=0}^k \left| z^n_t (f^\psi_l) - z^n_s (f^\psi_l )\right|^2 + 2  \sum_{l>k} \sup_{t \in [0,T]} \left|z^n_t (f^\psi_l)\right|^2.
	\end{equation}
	The first term can be make small by using the continuity of $z^n_t (e_l)$; for the second one, observe that we have just proven that $\bE \left[ \sum_{l\geq 1} \sup_{t \in [0,T]} \left|z^n_t (e_l)\right|^2 \right] < \infty$. This implies that there exists a subsequence $\{k_m\}_{m\in \N}$ such that $\sum_{l> k_m} \sup_{t \in [0,T]} \left|z^n_t (e_l)\right|^2$ tends to 0 almost surely as $m$ tends to infinity. The almost sure continuity in \eqref{z:cont} is then established by choosing $s$ and $t$ close enough and $k$ large enough.
	
\end{proof}

\subsection{Proof of Theorem \ref{thm:sup}} Recall that $\mu^n_0$ converges in $H_{-1}$ to $q_\psi \in M$. Next lemma assures that the projection of $\mu^n_0$ on $M$ is well defined for $n$ big enough.

\medskip

\begin{lemma}{\cite[Lemma 2.8]{cf:LP}}
	There exists $\sigma>0$ such that for all $\mu \in H_{-1}$ such that $\textup{dist} (\mu, M)\leq \sigma$, there exists a unique phase $\psi:= \textup{proj}_M (\mu) \in \bbT$ such that $P^0_\psi (\mu - q_\psi)=0$ and the mapping $\mu \mapsto \textup{proj}_M(\mu)$ is $\cC^\infty$. 
\end{lemma}

\medskip

Let $\psi_n = \textup{proj}_M(\mu^n_0)$. Fix $\varepsilon >0$, we place ourselves in 
\begin{equation}
\label{eq:A^n_1}
A^n_1 = \{ \norm{\mu^n_0 - q_{\psi_n}}_{-1} \leq \varepsilon/2 \}.
\end{equation}
Denote by $d^n_t$ the distance between $\mu^n_t$ and $M$, i.e.
\begin{equation}
	d^n_t := \textup{dist} (\mu^n_t, M).
\end{equation}
We want to prove that $\mu^n_t$ stays close to $M$ for long times. Let $T>0$, $N \in \N$ and define
\begin{equation}
T_i = iT, \quad \text{ for } i=0,\dots,N.
\end{equation}
If, for $N = N_n = o(\exp(n))$, we show that
\begin{equation}
	\lim_{n\to \infty } \bP \left(\sup_{t \in [0,T_{N_n}]} d^n_t \leq \varepsilon \right) =1,
\end{equation}
then we are done. For sake of notation, we just employ $N$.

\medskip

For $0 \leq a < b < \infty$, define the events
\begin{equation}
	E^n (a,b) = \left\{ \max \left\{2 d^n_a , \, 2d^n_b, \, \sup_{t \in (a,b)} d^n_t \right\} \leq \varepsilon \right\},
\end{equation}
clearly
\begin{equation}
	\bP \left(\sup_{t \in [0,T_N]} d^n_t \leq \varepsilon \right) \geq \bP \left( E^n(0,T_N)\right).
\end{equation}
The Markov property of system \eqref{eq:k} implies that
\begin{equation}
\begin{split}
	\bP \left( E^n(0,T_N) \right) \geq \, & \bP \left( E^n(0,T_N) \big| E^n(0,T_{N-1}) \right) \bP\left( E^n(0,T_{N-1}) \right) = \\
	=\, &\bP\left( E^n(0,T) \right) \bP\left( E^n(0,T_{N-1}) \right) \geq \, \bP \left( E^n(0,T) \right)^N.
\end{split}
\end{equation}

Let's then focus on the bounded interval of time $[0,T]$ and consider $\nu^n_\cdot= \mu^n_\cdot - q_{\psi_n}$, which satisfies the stochastic partial differential equation \eqref{eq:lMild2}. Taking the norm in $H_{-1}$ on both sides of \eqref{eq:lMild2} and using the properties of the semigroup together with the fact that (e.g. \cite[Lemma 7.3]{cf:BGP14})
\begin{equation}
\label{eq:estCon}
	\norm{\partial_\theta (\mu (J*\nu))}_{-2} \leq C \norm{\mu}_{-1} \norm{\nu}_{-1}, \quad \text{ for all } \mu, \nu \in H_{-1},
\end{equation}
one is left with (with a new constant $C$)
\begin{equation}
\label{eq:normSup}
	\norm{\nu^n_t}_{-1} \leq \norm{e^{tL_{\psi^n}}\nu^n_0}_{-1} + \int_0^t \frac C{\sqrt{t-s}} \norm{\nu^n_s}^2_{-1} \dd s + \norm{g^n_t}_{-1} + \norm{z^n_t}_{-1}.
\end{equation}
By taking $\varepsilon$ small enough, one can apply a Gronwall-type inequality (similar to Lemma \ref{lem:apriori}) that leads to (recall \eqref{eq:A^n_1} and the fact that the semigroup is continuous)
\begin{equation}
\sup_{t\in[0,T]} \norm{\nu^n_t}_{-1} \leq \frac 23 \varepsilon + \sup_{t\in[0,T]} \norm{g^n_t}_{-1} + \sup_{t\in[0,T]} \norm{z^n_t}_{-1}.
\end{equation}
For $\eta>0$, define $A^n_2 (\eta) = \{ \sup_{t\in[0,T]} \norm{z^n_t}_{-1} \leq \eta \}$. If $n$ is large enough, Lemma \ref{lem:g^n_t} assures that $\sup_{t\in[0,T]} \norm{g^n_t}_{-1}$ is arbitrarily small a.s., and, placing ourselves in $A^n_2 (\varepsilon/10)$, one obtains
\begin{equation}
\sup_{t\in[0,T]} \norm{\nu^n_t}_{-1} \leq \frac 23 \varepsilon + \frac 1{5}\varepsilon \leq \varepsilon.
\end{equation}
Plugging this estimate in \eqref{eq:normSup} for $t=T$, observing that $P^0_{\psi^n} \nu^n_0 =0$ by construction so that $\norm{e^{tL_{\psi^n}} \nu^n_0}_{-1} \leq C e^{-\lambda_1 t /2} \norm{\nu^n_0}_{-1}$ (e.g. \cite[Proposition B.6]{cf:LP}), one obtains
\begin{equation}
\norm{\nu^n_T}_{-1} \leq e^{-\lambda_1 T/2} \frac \varepsilon 2 + \varepsilon \left(\varepsilon \int_0^T \frac C{\sqrt{T-s}} \dd s \right) + \frac \varepsilon 5,
\end{equation}
choosing $T$ and $\varepsilon$ such that
\begin{equation}
\begin{split}
	& T \geq \frac 2 {\lambda_1}\log (5), \\
	& \varepsilon \leq \frac 1{20 C \sqrt{T}},
\end{split}
\end{equation}
one finally gets
\begin{equation}
	\norm{\nu^n_T}_{-1} \leq \frac \varepsilon 2.
\end{equation}
Since $d^n_t \leq \norm{\nu^n_t}_{-1}$ for all $t\geq 0$, we have then proven that $A^n_2 (\varepsilon/10) \subset E^n (0,T)$. In particular,
\begin{equation}
	\bP \left(\sup_{t \in [0,T_N]} d^n_t \leq \varepsilon \right) \geq \bP \left( E^n(0,T)\right)^N \geq \left(1- \bP \left( A^n_2(\epsilon/10)^\complement \right)\right)^N.
\end{equation}
We can than use the estimate \eqref{z:longTime} in Lemma \ref{lem:z^n_t} to get
\begin{equation}
	\bP \left( A^n_2(\epsilon/10)^\complement \right) = \bP \left( \sup_{t\in[0,T]} \norm{z^n_t}_{-1} > \frac{\varepsilon}{10} \right) \leq \exp \left\{ - \frac Z{100} n \varepsilon^2 \right\}.
\end{equation}

Putting all together, one is left with
\begin{equation}
\begin{split}
	\bP \left( \sup_{t \in [0,T_N]} d^n_t \leq \varepsilon \right) \geq \left(1- \bP \left( A^n_2(\epsilon/10)^\complement \right) \right)^N \geq \left(1 - \exp \left\{ - \frac Z{100} n \varepsilon^2 \right\} \right)^N \\
	= \exp \left\{ N \log \left[1 - \exp \left( - \frac Z{100} n \varepsilon^2 \right) \right] \right\} \geq \exp\left\{ -\frac 32 N \exp \left( - \frac Z{100} n \varepsilon^2 \right) \right\},
	\label{eq:finEst}
\end{split}
\end{equation}
where we have used that $\log(1-x) \geq - 3/2 x$ for $0 \leq x \leq 1/2$. But the right hand side of \eqref{eq:finEst} tends to 1 for all $N=N_n= o(\exp (n))$ and the proof is concluded.

\bigskip

\section{Longtime behavior around $1/2\pi$}
\label{s:longTimePi}
In this section we will suppose that the finite time behavior is already known, so that for $n$ large enough, $\mu^n_t$ is very close to $\mu_t$; thus, for a large $T_0$, $\mu_{T_0}$ will be very close to $1/2\pi$ and so will be for $\mu^n_{T_0}$. At the end of the day, we may suppose that we are starting close to $1/2\pi$. Since we are not assuming any independence between initial conditions and graph, instead of proving Theorem \ref{thm:main}, we rather prove the following Proposition.

\begin{proposition}
	\label{p:longTimes}
	If for every $\varepsilon_0 > 0$
	\begin{equation}
	\lim_{n\to \infty} \bP \left(\norm{\mu^n_0 - \tfrac 1{2\pi}}_{-1} \leq \varepsilon_0 \right) = 1.
	\end{equation}
	Then, there exists $A>0$ such that for every positive increasing sequence $\{T_n\}_{n\in\N}$ such that $T_n = \exp(o(n))$ and for all $0 < \varepsilon < A$, it holds
	\begin{equation}
	\lim_{n\to \infty} \bP \left(\sup_{t\in [0, T_n]} \norm{\mu^n_t - \tfrac 1{2\pi}}_{-1} \leq \varepsilon \right) = 1.
	\end{equation}	
\end{proposition}

%
%

The end of the section is thus devoted to prove Proposition \ref{p:longTimes}.

\subsection{A mild formulation around $1/2\pi$}
We place ourselves aroud the stationary solution $\tfrac 1{2\pi}$. The system evolution is captured by the linear dynamics around $\frac 1{2\pi}$ and the corresponding linear operator $L_{2\pi}$ is given by
\begin{equation}
\label{def:linOpPi}
L_{2\pi} u  \; := \; \tfrac 12 \partial^2_\theta u - \tfrac{1}{2\pi} (\partial_\theta J)* u, \quad \text{ for }u \in \cC^2(\bbT), \; \int_{\bbT} u(\theta) \dd \theta = 0.
\end{equation}
The adjoint $L^*_{2\pi}$ of $L_{2\pi}$ in $\cL^2_0$ has the following expression
\begin{equation}
\label{def:linOpPi*}
L^*_{2\pi} u  \; = \; \tfrac 12 \partial^2_\theta u - \tfrac{1}{2\pi} J*(\partial_\theta u),
\end{equation}
and domain $D(L^*_{2\pi})=D(L_{2\pi})$. These operators are diagonal in the Fourier basis $\{e_l\}_{l\geq1}$, with eigenvalues denoted by $\{\lambda^{2\pi}_l\}_{l\geq 1}$. The spectrum is negative and bounded away from 0, let $\gamma_K = \lambda^{2\pi}_1 = \frac{1-K}{2}>0$ denote the spectral gap. The operator $L_{2\pi}$ (resp. $L^*_{2\pi}$) defines an analytic semigroup $e^{tL_{2\pi}}$ (resp. $e^{tL^*_{2\pi}}$) with the following property:
\begin{equation}
\label{eq:linSemEst}
	\norm{e^{tL_{2\pi}} u}_{-1} \leq C \frac{e^{-\gamma t/2}}{\sqrt{t}} \norm{u}_{-2}, \quad\text{for some } C > 0,
\end{equation}
for all $\gamma \in [0, \gamma_K)$, all $t>0$ and $u \in H_{-1}$. We will not prove \eqref{eq:linSemEst} but refer to Appendix B for similar estimates.

\medskip

Define $\nu^n_t := \mu^n_t - \frac 1{2\pi}$. As done in Section \ref{s:longTimeM}, we derive a mild formulation for $\nu^n_\cdot$. We omit the proof.

\medskip

\begin{proposition}
	\label{p:mildLPi}
	The process $\nu^n_t \in H_{-1}$ satisfies the following stochastic partial differential equation in $C\left([0,T], H_{-1} \right)$:
	\begin{equation}
	\label{eq:lMildPi2}
	\nu^n_t = e^{t L_{2\pi}}\nu^n_0 - \int_{0}^{t} e^{(t-s)L_{2\pi}} \partial_\theta \left[\nu^n_s (J*\nu^n_s)\right] \dd s - g^n_t + z^n_t,
	\end{equation}
	where
	\begin{equation}
	g^n_t = \frac 1{n^2} \sum_{i,j=1}^{n} \int_0^t \left(\frac{\xi_{ij}}{p_n} - 1\right)  e^{(t-s)L_{2\pi}} \partial_\theta \left[\delta_{\theta^{i,n}_s} ( J*\delta_{\theta^{j,n}_s})\right] \dd s,
	\end{equation}
	and $z^n_t \in H_{-1}$ is defined for $h \in H_1$ by
	\begin{equation}
	\langle z^n_t, h \rangle_{-1,1} = \frac 1n \sum_{j=1}^n \int_0^t \partial_\theta e^{(t-s)L^*_{2\pi}}  h (\theta^{j,n}_s) \dd B^j_s.
	\end{equation}
\end{proposition}

\subsection{Control on the perturbations} Contrary to the supercritical case, the operator $L_{2\pi}$ is contracting along all direction or, in other words, all its eigenvalues are negative. This property gives a stronger control on $g^n_\cdot$ and $z^n_\cdot$, as shown in the next Lemmas.

\medskip

\begin{lemma}[Wellposedness and bounds on $g^n_t$]
	\label{lem:g^n_tPi}
	For $n\in\N$ and $t\geq 0$, let $g^n_t$ be given by
	\begin{equation}
	g^n_t = \frac 1{n^2} \sum_{i,j=1}^{n} \int_0^t \left(\frac{\xi_{ij}}{p_n} - 1\right)  e^{(t-s)L_{2\pi}} \partial_\theta \left[\delta_{\theta^{i,n}_s} ( J*\delta_{\theta^{j,n}_s})\right] \dd s.
	\end{equation}
	Then
	\begin{enumerate}
		\item $g^n \in \cC^0 ([0,\infty), H_{-1})$. In particular, for every $h\in H_1$ and $t \geq 0$
		\begin{equation}
		\label{g:equivPi}
		g^n_t (h) = - \frac 1{n^2} \sum_{i,j=1}^{n} \int_0^t \left(\frac{\xi_{ij}}{p_n} - 1\right) J(\theta^{i,n}_s- \theta^{j,n}_s)  \partial_\theta e^{(t-s)L^*_{2\pi}}  h (\theta^{i,n}_s) \dd s.
		\end{equation}
		\item There exists $D>0$, independent of $t$, such that
		\begin{equation}
		\label{g:supTPi}
		\norm{g^n_t}_{-1} \leq D \, \frac{\gnorm{P^{(n)}-\mathbf{1}^{(n)}}}{n^2}, \quad \text{ for all } t\geq 0.
		\end{equation}
	\end{enumerate}
\end{lemma}

\begin{proof}
	We only prove (2). Observe that, as in \eqref{g:L2sym},
	\begin{equation}
	\begin{split}
	\label{g:L2symPi}
	\langle \frac 1{n^2} \sum_{i,j=1}^{n} \left(\frac{\xi_{ij}}{p_n} - 1\right) e^{(t-s)L_{2\pi}} \partial_\theta \left[ \delta_{\theta^{i,n}_s} ( J*\delta_{\theta^{j,n}_s}) \right],   h \rangle_{-1,1} = \\
	= - \frac 1{n^2} \sum_{i,j=1}^{n} \left(\frac{\xi_{ij}}{p_n} - 1\right)   \langle \delta_{\theta^{i,n}_s},   ( J*\delta_{\theta^{j,n}_s}) \partial_\theta e^{(t-s)L^*_{2\pi}} h \rangle_{-1,1}.
	\end{split}
	\end{equation}
	Applying Theorem \ref{thm:gro}, this time with
	\begin{equation}
	\begin{split}
	&a_{ij} = \left(\tfrac{\xi_{ij}}{p_n} - 1\right), \\
	&S_i = \delta_{\theta^{i,n}_s}, \\
	&T_j = \frac{\sqrt{t-s}}{Ce^{-\gamma(t-s)}}\left(J*\delta_{\theta^{j,n}_s}\right) \partial_\theta e^{(t-s) L^*_{2\pi}} \, \frac h{\;\norm{h}_1},
	\end{split}
	\end{equation}
	allows us to bound the expression in \eqref{g:L2symPi} by 
	\begin{equation}
	K_R \frac{C e^{-\gamma(t-s)}}{\sqrt{t-s}} \norm{h}_1  \; \frac{\gnorm{P^{(n)}-\mathbf{1}^{(n)}}}{n^2}.
	\end{equation}
	This shows that
	\begin{equation}
	\norm{g^n_t}_{-1} \leq K_R C \frac{\gnorm{P^{(n)}-\mathbf{1}^{(n)}}}{n^2} \int_0^t \frac{e^{-\gamma(t-s)}}{\sqrt{t-s}}  \dd s \; \leq D \frac{\gnorm{P^{(n)}-\mathbf{1}^{(n)}}}{n^2} \; ,
	\end{equation}
	where $D :=  K_R C \int_0^\infty  \frac{e^{-\gamma s}}{\sqrt{s}} \dd s > 0$ since the integral converges.	The proof is concluded.
\end{proof}

\medskip

We now turn to the stochastic term $z^n_t$ in \eqref{eq:lMildPi2}. Recall that $L_{2\pi}$ is diagonal in the Fourier basis $\{e_l\}_{l\geq 1}$ of $H_{-1}$, with eigenvalues denoted by $\lambda^{2\pi}_l$. Then

\medskip

\begin{lemma}[Wellposedness and bounds on $z^n_t$]
	\label{lem:z^n_tPi}
	For $n\in\N$ and $t > 0$, let $z^n_t$ be defined by
	\begin{equation}
	z^n_t = \sum_{l\geq 1} \langle z^n_t, e_l\rangle_{H_{-1}} \; e_l,
	\end{equation}
	where
	\begin{equation}
	\langle z^n_t, e_l\rangle_{H_{-1}} = z^n_t \left(\frac {e^{il \cdot}}l\right) = \frac in \sum_{j=1}^n \int_0^t e^{(t-s)\lambda^{2\pi}_l} e^{il \theta^{j,n}_s} \dd B^j_s.
	\end{equation}
	Then
	\begin{enumerate}
		\item $z^n \in \cC^0 ([0,\infty), H_{-1})$ almost surely.
		\item There exists $C > 0$ independent of $n$, such that for all $T> 0$
		\begin{equation}
		\bE \left[ \sup_{t \in [0,T]} \norm{z^n_t}^2_{-1} \right] \leq C\; \frac{\log(1+2 \gamma_K T)}{n}.
		\end{equation}
		\item For every positive increasing sequence $\{T_n\}_{n\in\N}$ such that $T_n = \exp(o(n))$ and for all $\eta > 0$, it holds
		\begin{equation}
		\label{z:longTimePi}
		\lim_{n\to\infty} \bP \left( \sup_{t\in[0,T_n]} \norm{z^n_t}_{-1} \leq \eta \right) = 1.
		\end{equation}
	\end{enumerate}
\end{lemma}

\begin{proof}
	We only prove (2). For $l\geq 1$, let $x^l_t := \sqrt{2 \lambda^{2\pi}_l n} \, e^{\lambda^{2\pi}_l t} \left| z^n_t (e^{il\cdot}) \right|$. In particular
	\begin{equation}
	x^l_t = \left| \frac{\sqrt{2 \lambda^{2\pi}_l}}{\sqrt{n}} \sum_{j=1}^n \int_0^t e^{s\lambda^{2\pi}_l} e^{il\theta^{j,n}_s} \dd B^j_s \right| = \left|a^l_t+i\,b^l_t\right|,
	\end{equation}
	where $a^l$ and $b^l$ are two continuous real valued martingales. Let $\langle x^l \rangle_t = \langle a^l \rangle_t + \langle b^l \rangle_t$ where $\langle a^l \rangle_t$ and $\langle b^l \rangle_t$ are the quadratic variations of $a^l_t$ and $b^l_t$ respectively, then
	\begin{equation}
	\langle x^l \rangle_t = \frac{2 \lambda^{2\pi}_l}{n} \sum_{j=1}^n \int_0^t e^{2s\lambda^{2\pi}_l} (\cos^2 + \sin^2) (l\theta^{j,n}_s) \dd s = e^{2\lambda^{2\pi}_lt} -1.
	\end{equation}
	
	We now use
	\begin{lemma}
		\label{lem:PG}
		Let $Y_t = A_t + i \, B_t$, where $A_t$ and $B_t$ are continuous real valued martingales. Define $X_t = \left| Y_t \right|$ and $\langle X \rangle_t = \langle A \rangle_t + \langle B \rangle_t$, where $\langle A \rangle_t$ and $\langle B \rangle_t$ are the quadratic variations of $A$ and $B$ respectively. Then, there exists $C>0$ such that, for all $T>0$,
		\begin{equation}
		\label{eq:PG}
		\bE \left[ \sup_{t \in [0,T]} \frac{X^2_t}{1 + \langle X \rangle_t} \right] \leq C \log(1+\log(1+\langle X \rangle_t)).
		\end{equation}
	\end{lemma}
	The proof of Lemma \ref{lem:PG} is presented at the end of the section. By choosing $X_t = x^l_t$, $A_t = a^l_t$ and $B_t = b^l_t$, one obtains that, for $T>0$,
	\begin{equation}
	\bE \left[ \sup_{t \in [0,T]} \left| z^n_t (e_l) \right|^2 \right] = \frac{1}{2\lambda^{2\pi}_ln} \, \bE \left[ \sup_{t \in [0,T]} \frac{(x^l_t)^2}{1 + \langle x^l \rangle_t} \right] \leq \frac C{2\lambda^{2\pi}_l n} \log (1+ 2 \lambda^{2\pi}_l T).
	\end{equation}
	
	It remains to observe that
	\begin{equation}
	\bE \left[ \sup_{t \in [0,T]} \norm{z^n_t}^2_{-1} \right] \leq \bE \left[  \sum_{l\geq 1}  \sup_{t \in [0,T]} \left|z^n_t(e_l)\right|^2 \right] \leq C \sum_{l\geq 1} \frac 1{2\lambda^{2\pi}_l n} \log (1+ 2 \lambda^{2\pi}_l T).
	\end{equation}
	The conclusion holds by factorizing the first term of the sum and modifying the constant $C$ accordingly: observe that $\sum_{l\geq1} \sup_{T\geq 1} \tfrac {\log (1+2\lambda^{2\pi}_l T)}{\lambda^{2\pi}_l \log (1+ 2\lambda^{2\pi}_1 T)} < \infty$.

	The proof is concluded modulo Lemma \ref{lem:PG}, proven hereafter.
\end{proof}

\medskip

\begin{proof}[Proof of Lemma \ref{lem:PG}]
	Recall that $A_t$ is a martingale, in particular a slight variation of \cite[Corollary 2.8]{cf:GP} implies that there exists $D>0$ such that
	\begin{equation}
	\bE \left[ \sup_{t \in [0,T]} \frac{A^2_t}{1 + \langle A \rangle_t} \right]  \leq  D \log(1+\log(1+\langle A \rangle_t)).
	\end{equation}
	Thus, one can develop
	\begin{equation}
	\begin{split}
	\bE \left[ \sup_{t \in [0,T]} \frac{X^2_t}{1 + \langle X \rangle_t} \right]&  \leq \bE \left[ \sup_{t \in [0,T]} \frac{A^2_t}{1 + \langle A \rangle_t} \right] + \bE \left[ \sup_{t \in [0,T]} \frac{B^2_t}{1 + \langle B \rangle_t} \right] \leq \\
	\leq &D \log(1+\log(1+\langle A \rangle_t)) + D \log(1+\log(1+\langle B \rangle_t)) \leq \\
	\leq  & 2D \log(1+\log(1+\langle X \rangle_t)),
	\end{split}
	\end{equation}
	and the proof is done by taking $C=2D$.
\end{proof}

\subsection{Proof of Proposition \ref{p:longTimes}}
	Fix $\varepsilon>0$. From Proposition \ref{p:mildLPi} we know that $\nu^n_t := \mu^n_t - \tfrac 1{2\pi}$ satisfies
	\begin{equation}
	\label{eq:lmild2}
	\nu^n_t = e^{t L_{2\pi} }\nu^n_0 - \int_{0}^{t} e^{(t-s)L_{2\pi}} \partial_\theta \left[ \nu^n_s (J*\nu^n_t) \right]\dd s - g^n_t + z^n_t.
	\end{equation}
	Taking the norm and using the properties of $e^{tL_{2\pi}}$, together with the estimate \eqref{eq:estCon}, for all $0<\gamma<\gamma_K$ one obtains
	\begin{equation}
	\label{eq:lnorm}
	\norm{\nu^n_t}_{-1} \leq \norm{\nu^n_0}_{-1} + C \int_{0}^{t} \frac{e^{- \gamma (t-s)}}{\sqrt{t-s}} \norm{\nu^n_s}^2_{-1} \dd s + \norm{g^n_t}_{-1} + \norm{z^n_t}_{-1}.
	\end{equation}
	
	\medskip
	
	Thanks to the contractive properties of $L_{2\pi}$, there exists $D>0$ (Lemma \ref{lem:g^n_tPi}) such that
	\begin{equation}
	\sup_{t\geq 0} \norm{g^n_t}_{-1} < D \, \frac{\gnorm{P^{(n)} - \mathbf{1}^{(n)}}}{n^2}.
	\end{equation}
	Define now $B^n_1 (\varepsilon_0) = \{ \norm{\nu^n_0} \leq \varepsilon_0 \}$ and $B^n_2 (\eta) = \{ \sup_{t \in[0,T_n ]} \norm{z^n_t}_{-1} \leq \eta \}$. On $B^n_1 (\varepsilon/3) \cap B^n_2 (\varepsilon/4)$ and for $n$ large enough, we can apply Lemma \ref{lem:apriori} with
	\begin{equation}
	\begin{split}
	& \delta = \frac{\varepsilon}{3}, \qquad T = T_n, \\
	& f(t) = \norm{\nu^n_t}_{-1}, \\
	& g(t) = \norm{g^n_t}_{-1} + \norm{z^n_t}_{-1},
	\end{split}
	\end{equation}
	and obtain
	\begin{equation}
	\sup_{t \in [0,T_n]} \norm{\nu^n_t}_{-1} \leq \varepsilon.
	\end{equation}
	
	The proof is concluded since by hypothesis $\bP (B^n_1) \to 1$ and Lemma \ref{lem:z^n_tPi} implies that $\bP (B^n_2) \to 1$ as $n$ tends to infinity.

\section{Finite time behavior}
\label{s:finTim}
The aim of this section is to study the closeness of $\mu^n_\cdot$ to $\mu_\cdot$ on bounded time interval. 

\begin{proof}[Proof of Theorem \ref{thm:finTim}]
Fix $\varepsilon>0$ and $T>0$. It is not difficult to see that $\mu^n_\cdot - \mu_\cdot$ satisfies again a mild equation in $\cC^0 ([0,T], H_{-1})$, which is given by
\begin{equation}
\label{eq:mild2}
	\mu^n_t - \mu_t = e^{t \frac{\Delta}{2}}\left( \mu_0^n - \mu_0 \right) - \int_{0}^{t} e^{(t-s)\frac{\Delta}{2}} \partial_\theta \left[ \mu^n_s (J*\mu^n_s) - \mu_s(J*\mu_s) \right]\dd s - g^n_t + z^n_t,
\end{equation}
where
\begin{equation}
	g^n_t = \frac 1{n^2} \sum_{i,j=1}^{n} \int_0^t \left(\frac{\xi_{ij}}{p_n} - 1\right)  e^{(t-s)\frac{\Delta}2} \partial_\theta \left[\delta_{\theta^{i,n}_s} ( J*\delta_{\theta^{j,n}_s})\right] \dd s,
\end{equation}
and $z^n_t$ is denoted for $h \in H_1$ by
\begin{equation}
	z^n_t (h) = \frac 1n \sum_{j=1}^n \int_0^t \partial_\theta e^{(t-s)\frac{\Delta}2}  h (\theta^{j,n}_s) \dd B^j_s.
\end{equation}
Observe that we are using the Laplacian operator which is very similar to $L_{2\pi}$ except for the first eigenvalue that is now given by $-(1-K)/2$. We will thus use all the results about $L_{2\pi}$ and its semigroup to control $z^n_\cdot$ and $g^n_\cdot$.
\medskip

Taking the $H_{-1}$ norm in \eqref{eq:mild2} and applying \eqref{eq:estCon}, one is left with
\begin{equation}
\label{eq:norm}
\norm{\mu^n_t - \mu_t}_{-1} \leq \norm{\mu_0^n - \mu_0}_{-1} + \int_{0}^{t} \tfrac{C}{\sqrt{t-s}} \norm{\mu_s^n - \mu_s}_{-1} \dd s + \norm{g^n_t}_{-1} + \norm{z^n_t}_{-1}.
\end{equation}
The term involving the graph $g^n_t$  can be controlled again by $\gnorm{P^{(n)} - \mathbf{1}^{(n)}}$: minor modifications to Lemma \ref{lem:g^n_tPi} show that there exists $D>0$ such that
\begin{equation}
\sup_{t \in [0,T]} \norm{g^n_t}_{-1} \leq D \, \frac{\gnorm{P^{(n)} - \mathbf{1}^{(n)}}}{n^2}.
\end{equation}
For the initial conditions and the stochastic part $z^n_t$, define the two sets:
\begin{equation}
\begin{split}
C^n_1 = C^n_1 (\varepsilon_0) = \left\{ \norm{\mu_0^n - \mu_0}_{-1} \leq \varepsilon_0 \right\}; \\
C^n_2 = C^n_2 (T, \eta) = \left\{ \sup_{t\in[0,T]} \norm{z^n_t}_{-1} \leq \eta \right\}.
\end{split}
\end{equation}
On $C^n_1 \cap C^n_2$, one obtains
\begin{equation}
\norm{\mu^n_t - \mu_t}_{-1} \leq \varepsilon_0 + \int_{0}^{t} \tfrac{C}{\sqrt{t-s}} \norm{\mu_s^n - \mu_s}_{-1} \dd s + D \, \frac{\gnorm{P^{(n)} - \mathbf{1}^{(n)}}}{n^2} +  \eta.
\end{equation}
Gronwall-Henry's inequality (\cite[Lemma 7.1.1 and Exercice 1]{cf:henry}) leads to 
\begin{equation}
\sup_{t\in[0,T]} \norm{\mu^n_t - \mu_t}_{-1} \leq  2 \left( \varepsilon_0 + D \, \frac{\gnorm{P^{(n)} - \mathbf{1}^{(n)}}}{n^2} + \eta \right) e^{aT},
\end{equation}
where $a$ is independent of $n$, $\varepsilon_0$ and $\eta$. Considering $\varepsilon_0$ and $\eta$ small enough and $n$ large enough, the proof is concluded modulo showing that
\begin{equation}
\lim_{n\to \infty} \bP\left( C^n_1 \cap C^n_2\right) = 1.
\end{equation}
From the hypothesis on the intial condition \eqref{h:mu0}, it is clear that for all $\varepsilon_0$ one has $\bP \left( C^n_1 (\varepsilon_0) \right) \to 1$ as $n$ tends to infinity. The same conclusion holds for $C^n_2$ by slightly modifying the proof of Lemma \ref{lem:z^n_tPi}. The proof is concluded.
\end{proof}

\appendix
\section{Graphs}

\subsection{General properties of the graphs under consideration}\label{sub:graph} We observe that condition \eqref{h:graph} implies a weak form of degree homogeneity (recall \eqref{h:gnorm}):
\begin{lemma}\label{lem:hom}
	Suppose that \eqref{h:graph} holds. Let $\delta > 0$, define
	\begin{equation}
	I^\delta_n := \left\{ i \in \{1, \dots, n\} : \lim_{n\to \infty}  \left|\frac 1n \sum_{j=1}^{n} \frac{\xi^{(n)}_{i,j}}{p_n} - 1\right| \geq \delta  \right\}.
	\end{equation}
	Then $|I^\delta_n| = o(n)$.
\end{lemma}

\begin{proof}
	Suppose that $\lim_{n\to \infty} \frac{|I_n|}{n} = c$ for some $c >0$. Then
	\begin{equation}
	\begin{split}
	\sup_{s_i,t_j \in \{\pm 1\} } \, \frac 1 {n^2} \sum_{i,j=1}^n \left(\frac{\xi_{ij}^{(n)}}{p_n} - 1 \right)s_i t_j \geq 	\sup_{s_i \in \{\pm 1\} } \, \frac 1 n \sum_{i=1}^n \left[ \frac 1 n \sum_{j=1}^n \left(\frac{\xi_{ij}^{(n)}}{p_n} - 1 \right)\right]s_i \geq \\
	\geq \frac 1n \sum_{i \in I_n} \left| \frac 1 n \sum_{j=1}^n \left(\frac{\xi_{ij}^{(n)}}{p_n} - 1 \right)\right| \geq \, \frac{|I_n|}{n} \, \inf_{i \in I_n}  \left| \frac 1 n \sum_{j=1}^n \left(\frac{\xi_{ij}^{(n)}}{p_n} - 1 \right)\right|.
	\end{split}
	\end{equation}
	This last term does not go to zero as $n$ tends to infinity, against \eqref{h:graph}.
\end{proof}

\medskip

It also implies the existence of an unique giant component.

\begin{lemma}
	\label{lem:connect}
	Suppose that \eqref{h:graph} holds. Then, there exists a unique sequence of connected components $\{\cC^{(n)}\}$ in $\{\xi^{(n)}\}$ and $\lim_{n\to \infty} \left| \cC^{(n)} \right|/n = 1$.
\end{lemma}
\begin{proof}
	We prove the uniqueness first. Suppose that for every $n$ there exist $\cC^{(n)}_1$ and $\cC^{(n)}_2$ distinct connected components of $\xi^{(n)}$ such that $\left| \cC^{(n)}_i \right| = n_i = \Theta(n)$ for $i=1,2$. Without loss of generality, one can suppose $\cC^{(n)}_1$ consisting in the first $n_1$ vertices of $\xi^{(n)}$ and $\cC^{(n)}_2$ in the following $n_2$. 
	
	Using the equivalence of $\ell_\infty \to \ell_1$ norm with the cut-norm (e.g. \cite{cf:AN06}), one obtains
	\begin{equation}
		\gnorm{P_n - \mathbf{1}_n} \geq \sup_{x_i, y_j \in \{0,1\}} \left|\sum_{i,j=1}^n \left(\frac{\xi_{ij}}{p_n} - 1\right) x_i y_j\right| \geq \sum_{\substack{1\leq i \leq n_1 \\ n_1 \leq j \leq n_2 - n_1}} 1 \, =  \, n_1 n_2 = \Theta(n^2).
	\end{equation}
	
	For the existence, suppose the connected components of $\xi^{(n)}$ are ordered from the biggest one in size (the first $n_1$ vertices) to the smallest one (the last vertices). Take the first $m$ components such that $\left| \cC_1 \cup \dots \cup \cC_m \right| \geq n/4$. One easily sees that $\left| \cC_1 \cup \dots \cup \cC_m \right| \leq n/2$. Applying the same reasoning of before with $1\leq i \leq n/4$ and $n/2 \leq j \leq n$, the proof is concluded.
\end{proof}

\medskip

\subsection{Examples of graph sequences}
We exhibit two classes of graphs, a random and a deterministic one, that satisfy assumption \eqref{h:graph}. The only hypothesis required on $p_n$ is equivalent to asking that the mean degree per site diverges as $n$ tends to infinity, i.e. $np_n \uparrow \infty$.

\medskip

\subsubsection{Erd\H{o}s-Rényi random graphs}
As mentioned in the introduction, $\gnorm{\cdot}$ has been found very useful for random graph concentration and this is indeed the case of ER graphs (e.g. \cite{cf:GV}). We recall the definition and give the result.

\medskip

For every $n\in\N$, let $\{\xi^{(n)}_{ij}\}_{1 \leq i \neq j \leq n}$ be IID Bernoulli random variables with parameter $p_n$, $\bbP$ denoting the associated probability. For every $i$, $\xi_{ii}^{(n)}$ is set equal to 0, i.e. self loop are not admitted.

\begin{lemma}
	\label{lem:ber}
	Assume that
	\begin{equation}
	\label{h:np_n}
	\lim_{n\to\infty} np_n = \infty.
	\end{equation}
	There exists $n_0 \in \N$ such that
	\begin{equation}
	\label{eq:ber}
	\bbP \left( \sup_{s_i, t_j} \;  \frac{1}{n^2} \sum_{i,j=1}^n \left(\tfrac{\xi_{ij}}{p_n} - 1 \right)s_i t_j  \geq \frac{2}{\sqrt{np_n}} \right) \leq e^{-2n}, \quad \text{for all } n\geq n_0.
	\end{equation}
\end{lemma}
\begin{proof}
	The proof is just an union bound and an application of Bernstein's inequality. Indeed, 
	\begin{equation}
	\label{eq:unionBound}
	\bbP \left( \sup_{s_i, t_j} \;  \frac{1}{n^2} \sum_{i,j=1}^n \left(\frac{\xi_{ij}}{p_n} - 1 \right)s_i t_j  \geq \frac{\delta}{\sqrt{np_n}} \right) \leq \sum_{s_i,t_j} 	\bbP \left( \frac{1}{n^2} \sum_{i,j=1}^n \left(\frac{\xi_{ij}}{p_n} - 1 \right)s_i t_j  \geq \frac{\delta}{\sqrt{np_n}} \right).
	\end{equation}
	
	Bernstein's inequality (\cite[Corollary 2.11]{cf:bouche}) says that if $X_1, \dots, X_n$ are independent zero-mean random variables such that $|X_j| \le M$ a.s. for all $j$, then for all $t\ge 0$
	\begin{equation*}
	\bbP \left( \sum_{j=1}^{n} X_j > t \right) \,\le\,  \exp\left\{-\frac{ t^2}{2 \sum_{j=1}^n \bbE[X_j^2] + \frac 23 Mt} \right\}.
	\end{equation*}
	Let $X_{k(i,j)} = \tfrac{s_i t_j}{n^2p_n}\left(\xi_{ij} - p_n \right)$ with $k$ some bijection from $\{1,\dots, n\}^2$ to $\{1,\dots,n^2\}$. Then $\left|X_k\right| \leq \tfrac{1}{n^2p_n}$ and $\bbE \left[ X_k^2 \right] \leq \tfrac{2}{n^4}$. For $n$ large enough, we thus obtain
	\begin{equation}
	\bbP \left( \sum_{k=1}^{n^2} X_k \geq \frac{\delta}{\sqrt{np_n}}\right) \le \exp\left\{ - \frac{n \delta^2}{4p_n + \frac 23 \frac{\delta}{\sqrt{np_n}}} \right\} \leq \exp\left\{ -  n \delta^2 \right\} .
	\end{equation}
	The proof is concluded observing that the sum in \eqref{eq:unionBound} consists in $4^n$ elements and choosing $\delta = 2$.
\end{proof}

We thus have

\begin{proposition}
	\label{p:er}
	Given \eqref{h:np_n}, ER graphs satisfy condition \eqref{h:graph}  $\bbP$-almost surely.
\end{proposition}
\begin{proof}
	It suffices to apply Borel-Cantelli lemma to \eqref{eq:ber}.	
\end{proof}

Similarly one can prove that symmetric ER random graphs satisfy \eqref{h:graph} a.s..

\medskip

\subsubsection*{Ramanujan graphs}
Let $d=2,3, \dots$, consider a $d$-regular graph, i.e. graph where each vertex has exactly $d$ neighbors. We start recalling a well-known result

\begin{lemma}[Expander mixing lemma]
	Let $G$ be a $d$-regular random graph ($G$ denoting the adjacency matrix itself), it holds
	\begin{equation}
	\label{eq:eml}
	\frac{1}{n^2} \gnorm{\tfrac nd G - \mathbf{1}^{(n)}} \leq 4 \, \frac{\lambda (d)}{d},
	\end{equation}
	where $\lambda (d)$ is the second biggest eigenvalue (in absolute value) associated to $G$.
\end{lemma}

\begin{proof}
	The proof is classical but it is in general formulated in terms of the cut-norm (e.g. \cite{cf:HNW06}). One easily sees that the cut-norm is equivalent (paying a factor 4, e.g. \cite{cf:AN06}) to the $\ell_\infty \to \ell_1$ norm.
\end{proof}

Ramanujan graphs are $d$-regular graphs such that $\lambda(d) \leq 2 \sqrt{d-1}$, they are very well known for their expander properties (e.g. \cite{cf:HNW06}). Condition \eqref{h:graph} holds whenever $d_n$ diverges; indeed
\begin{proposition}
	\label{p:ram}
	Let $d_n = np_n$. Suppose that \eqref{h:np_n} holds, i.e.
	\begin{equation}
	\lim_{n\to \infty} d_n = \infty.
	\end{equation}
	Then, every sequence of Ramanujan graphs satisfies condition \eqref{h:graph}.
\end{proposition}
\begin{proof}
	Rewriting \eqref{eq:eml} in terms of $p_n$, it becomes
	\begin{equation}
	\frac{1}{n^2} \gnorm{\tfrac G{p_n} - \mathbf{1}^{(n)}} \leq \, \frac 8{\sqrt{np_n}}.
	\end{equation}
	The proof is concluded taking the limit for $n$ which tends to infinity.
\end{proof}

\medskip

\subsection{Links with graphons}
The norm $\norm{\cdot}_{\infty \to 1}$ is strictly related to the canonical distance $d_{\cW}$ on the space of (sparse) graphons $\cW$ (e.g. \cite{cf:graphon}). In fact, whenever $\xi^{(n)}/p_n$ is (a realization of) a graphon $W^{(n)}$, condition \eqref{h:graph} is implied by the convergence of $W^{(n)}$  to the constant graphon $W\equiv 1$ in $\cW$. One can then consider system \eqref{eq:k} on a sequence of (sparse) graphons and require, instead of condition \eqref{h:graph}, the convergence in $\cW$ to the constant graphon.

We have decided not to add another level of complexity in order to keep the results as clear as possible, but everything could be reformulated within this more general framework and the proofs would basically not change.


\section{$H_{-1}$ and Semigroups}
\subsection{On the relationship between $H_{-1}$ and $\cP(\bbT)$}
Consider $H_1 := H_{1,1}$, its dual space, denoted by $H_{-1}$, can be described through the Fourier orthonormal basis $\left\{e_l \right\}_{l \geq 1}$, where $e_l(\theta) = l e^{il\theta}$. With this characterization one easily obtains that $\cP(\bbT) - \frac 1 {2\pi}\subset H_{-1}$. Indeed, for $\mu \in \cP (\bbT)$,
\begin{equation}
\label{eq:H1_inject}
\norm{\mu - \frac 1 {2\pi}}_{-1} = \sqrt{\sum_{l\geq 1} \left|\langle \mu, le^{il\cdot}\rangle_{H_{-1}} \right|^2} = \sqrt{\sum_{l\geq 1} \frac 1{l^2} \left|\langle \mu, e^{il\cdot}\rangle \right|^2}   \leq \sqrt{\sum_{l\geq 1}\frac 1{l^2}} < \infty.
\end{equation}
In particular, the difference between two probability measures belongs to $H_{-1}$. 
\medskip

Observe now that $H_{-1}$ induces a distance on $\cP (\bbT)$ which controls the bounded-Lipschitz distance $d_{\text{bL}}$, i.e. for all $\mu, \nu \in \cP(\bbT)$
\begin{equation}
\begin{split}
d_{\text{bL}} (\mu, \nu) = & \sup_{\norm{f}_\text{bL} = 1} \int f \left(\dd \mu - \dd \nu\right) \leq \sup_{h \in \cC^1_0, \norm{h}_1 = 1} \int h \left(\dd \mu - \dd \nu\right) = \\
= & \sup_{h \in \cC^1_0, \norm{h}_1 = 1} \int h' \left(\mathcal{U} - \mathcal{V}\right) = \sup_{\norm{h}_1 = 1} \langle \mu - \nu , h \rangle_{-1,1} = \\
= & \norm{\mu-\nu}_{-1}.
\end{split}
\end{equation}
Where we have used the density of $\cC^1_0$ in $H_1$, and denoted by $\mathcal U$ and $\mathcal V$ the primitives of $\mu$ and $\nu$ respectively.

\subsection{On the weighted Hilbert space $H_{-1,\omega}$}
Recall that, one has this sequence of continuous and dense inclusions:
\begin{equation}
H_{1,1/\omega} \subset \cL^2_0 = {\cL^2_0}^* \subset H_1^* =: H_{-1, \omega},
\end{equation}
where we have chosen the canonical identification for $\cL^2_0$. We can explicit the isometry between $H_{1,1/\omega}$ to $H_{-1,\omega}$. Consider the operator
\begin{equation}
\begin{split}
\label{d:isometry}
	A_{\omega}: \cC^\infty (\bbT)& \rightarrow \cC^\infty(\bbT)\\
	f& \mapsto -\partial_\theta \left(\omega^{-1} \, \partial_\theta f\right)
\end{split}
\end{equation}
It is known \cite[pag. 82]{cf:brezis} that $A_{\omega}(H_{1,1/\omega})$ is dense in $H_{-1,\omega}$ and the injection is continuous. This allows considering $H_{1, 1/\omega}$ as a subset of $H_{-1, \omega}$ by identifying $u$ and $A_\omega u$.

The inner product in $H_{-1,\omega}$, dual to the one in $H_{1,1/w}$, is given by
\begin{equation}
\langle u, v \rangle_{H_{-1,w}} = \int w \, \mathcal{U} \mathcal{V},
\end{equation}
where $\mathcal{U}$ and $\mathcal{V}$ are primitive of $u$ and $v$ respectively, such that $\int w \, \mathcal{U} = 0 = \int w \, \mathcal{V}$ (e.g. \cite[Subsection 2.2]{cf:BGP10}). Then, for $f,g \in \cC^\infty$, it holds
\begin{equation}
\langle A_{\omega} f, A_{\omega} g \rangle_{-1, \omega} = \int \omega^{-1} f' g' = \langle f, g \rangle_{1, 1/\omega}.
\end{equation}

\bigskip

\subsection{The linear operators $L_\psi$ and $L^*_\psi$ and their semigroups}
This subsection recalls the known results on $L_\psi$, its dual $L^*_\psi$ and the associated semigroups $e^{tL_\psi}$ and $e^{tL^*_\psi}$.
\medskip

We start with the spectral properties of $L_\psi$.
\medskip

\begin{proposition}
	\label{p:linOp}
	The operator $L_\psi$ (resp. $L^*_\psi$) is essentially self-adjoint with compact resolvent in $H_{-1,1/q}$ (resp. $H_{1,q}$). Its spectrum is pure point and lies in $(-\infty, - \lambda_1]\cup \{0\}$, where $\lambda_1>0$ and 0 is a simple eigenvalue of $L_\psi$ with eigenvector $\partial_\theta q_\psi$.
	
	Moreover, both $L_{2\pi}$ and $L^*_{2\pi}$ generate a $\cC^0$ semigroup $t\mapsto e^{tL_\psi}$ (resp. $t \mapsto e^{tL^*_\psi}$) in $\cL^2_0$ and $e^{tL^*_\psi} = \left(e^{tL_\psi}\right)^*$.
\end{proposition}
\begin{proof}
	The result about $L_\psi$ is given in \cite{cf:BGP10}.  Observe that, due to the isometry \eqref{d:isometry} between $H_{-1,1/q_\psi}$ and $H_{1,q_\psi}$, $L^*_\psi = A^{-1}_{1/q_\psi} L_\psi A_{1/q_\psi}$ and it has thus the same spectral properties of $L_\psi$.
	
	From the spectral properties of $L_\psi$ and $L^*_\psi$, one deduces that the two operators are sectorial (and with dense domain in $H_{-1}$), standard techniques assure the existence of the analytic semigroup (e.g. \cite{cf:henry}).
\end{proof}

\medskip

An accurate analysis of the semigroup has already been established in \cite{cf:BGP14} by means of interpolating norms and Fourier decomposition. We recall here the most important properties. We will use the space $H_{-2}$, defined in an analogous way of $H_{-1}$.

\medskip

\begin{proposition}[{\cite[Lemma 7.2]{cf:BGP14}}]
	\label{p:linSem}
	For all $t>0$, the operator $e^{t L_\psi}$ extends to a bounded operator from $H_{-2}$ to $H_{-1}$ and there exists $C>0$ such that for all $u \in H_{-2}$
	\begin{equation}
	\norm{e^{tL_\psi} u}_{-1} \leq C \left( 1+ \frac 1{\sqrt{t}} \right) \norm{u}_{-2}.
	\end{equation}
	Moreover, for all $\epsilon \in (0,1/2)$, $\delta\geq 0$ and all $u \in H_{-1}$
	\begin{equation}
	\norm{e^{(t+\delta)L_\psi} u - e^{tL_\psi}u}_{-1} \leq C \delta^\epsilon \left( 1+ \frac 1{t^{1/2+\epsilon}} \right) \norm{u}_{-2}.
	\end{equation}
	
	By duality, observe that for all $h \in H_1$
	\begin{equation}
		\norm{e^{tL^*_\psi} h}_{2} \leq C \left( 1+ \frac 1{\sqrt{t}} \right) \norm{h}_{1}.
	\end{equation}
\end{proposition}

\medskip

We end this subsection with an useful result on the eigenvalues and eigenfunctions associated to $L_\psi$, recall \eqref{d:eig}.

\begin{proposition}
	\label{p:eigFun}
	There exists $C>1$ such that for all $l\in \N$
	\begin{equation}
		\frac{l^2}{C} \leq \lambda_l \leq C l^2.
	\end{equation}
	
	Let $f^\psi_l = A^{-1}_{1/q_\psi} e^\psi_l$, then $f^\psi_l$ is an eigenfunction of $L^*_\psi$ associated to $-\lambda_l$ and
	\begin{equation}
		\sup_{l\in \N} \norm{\partial_\theta f^\psi_l}_\infty < \infty.
	\end{equation}
\end{proposition}
\begin{proof}
	The first part is covered in \cite[Remark 8.3]{cf:BGP14} and the second one in \cite[Corollary 8.6]{cf:BGP14}.
\end{proof}

\medskip

\subsection{Analytical estimate}
A variation on Gronwall Lemma.

\medskip

\begin{lemma}
	\label{lem:apriori}
	Let $T>0$, $\gamma\geq0$. Let $f: [0,T] \rightarrow [0,\infty)$ be a continuous function and $g:[0,T] \rightarrow [0,\infty)$ be such that for all $0\leq t \leq T$
	\begin{equation}
	f(t) \leq f(0) + \int_0^t \frac{e^{-\gamma(t-s)}}{\sqrt{t-s}} f^2(s) \dd s + g(t).
	\end{equation}
	There exists $A>0$, which depends on $T$ only if $\gamma=0$, such that for all $0 < \delta < A$ and if $f(0)< \delta $, $\sup_{t\in[0,T]} g (t)< \delta$, then
	\begin{equation}
	\sup_{t \in [0,T]} f(t) \leq 3\delta.
	\end{equation}
\end{lemma}

\begin{proof}
Consider the set $O = \{ t \; : \; f(t) \leq 3\delta \} \subset [0,T]$. Since $f$ is continuous and $f(0) \leq \delta$, $O$ is a non-empty open set in $[0,T]$. Suppose that $\sup (O) = u < T$; we show that $u \in O$, which implies $O = [0,T]$. 

Consider
\begin{equation}
\begin{split}
f(u) = & f(0) + \int_0^u \frac{e^{-\gamma(u-s)}}{\sqrt{u-s}} f^2(s) \dd s + g(u) \leq \\
\leq & 2 \delta + \delta \left( 9\delta \int_0^u \frac{e^{-\gamma(u-s)}}{\sqrt{u-s}} \dd s \right) \leq 3\delta,
\end{split}
\end{equation}
where the last inequality holds for all $\delta \leq A := \left( 9 \int_0^\infty \frac{e^{-\gamma s}}{\sqrt{s}} \dd s \right)^{-1}$ whenever $\gamma>0$ or for all $\delta \leq \frac 1{18\sqrt{T}}$ in case $\gamma=0$. Thus $u \in O$ and the proof is concluded.
\end{proof}

%

\bigskip

\section*{Acknowledgments} 
The author is thankful to Giambattista Giacomin for proposing this subject and his insightful advises and to Helge Dietert for the help on the analytic part. He would also thank Florian Bechtold, Simon Coste, Christophe Poquet and Assaf Shapira for discussions and comments on the previous drafts of the work.

The author acknowledges the support from the European Union’s Horizon 2020 research and innovation programme under the Marie Sk\l odowska-Curie grant agreement No 665850.

\end{document}